\documentclass{article}

\usepackage{a4wide}
\usepackage{graphicx}
\usepackage{amsmath}
\usepackage{amssymb}
\usepackage{amsthm}
\usepackage{enumerate}
\usepackage{verbatim}
\usepackage{hyperref}

\newtheorem{theorem}{Theorem}[section]
\newtheorem{definition}[theorem]{Definition}

\newtheorem{lemma}[theorem]{Lemma}
\newtheorem{claim}[theorem]{Claim}
\newtheorem{problem}[theorem]{Problem}
\newtheorem{corollary}[theorem]{Corollary}
\newtheorem{conjecture}[theorem]{Conjecture}

\usepackage[margin=20pt,small,bf]{caption}

\newcommand{\cCycles}{43000}
\begin{document}

\title{Partitioning a graph into a cycle and a sparse graph}

\author{\large{Alexey Pokrovskiy}\footnote{Email: 
\texttt{DrAlexeyPokrovskiy@gmail.com}} 
\\
\\ ETH Z\"urich,\\
Z\"urich, Switzerland
\\ 
\\ \small Keywords: Partitioning graphs, Ramsey theory, cycles.}

\maketitle

\begin{abstract}
In this paper we investigate results of the form ``every graph $G$ has a cycle $C$ such that the induced subgraph of $G$ on $V(G)\setminus V(C)$ has small maximum degree.'' Such results haven't been studied before, but are motivated by the  Bessy and Thomass\'e Theorem which states that the vertices of any graph $G$ can be covered by a cycle $C_1$ in $G$ and disjoint cycle $C_2$ in the complement of $G$.

There are two main theorems in this paper. The first is that every graph has a cycle with  $\Delta(G[V(G)\setminus V(C)])\leq \frac12(|V(G)\setminus V(C)|-1)$. The bound on the maximum degree $\Delta(G[V(G)\setminus V(C)])$ is best possible. The second theorem is that  every $k$-connected graph $G$  has a cycle with  $\Delta(G[V(G)\setminus V(C)])\leq \frac1{k+1}|V(G)\setminus V(C)|+3$.
We also give an application of this second theorem to a conjecture about partitioning edge-coloured complete graphs into monochromatic cycles.
\end{abstract}

\section{Introduction}\label{SectionIntroduction}
The goal of this paper is to investigate  results of the following form.
\begin{equation}\label{MainQuestion}\tag{$\dagger$}
\parbox[b]{0.9\textwidth}{``Every graph $G$ has a cycle $C$ such that the induced subgraph of $G$ on $V(G)\setminus V(C)$ has small maximum degree.''}
\end{equation}

In other words, does every graph have a cycle $C$, such that the subgraph of $G$ outside $C$ is sparse?
For convenience, throughout the paper we will use ``$G\setminus C$'' to mean the induced subgraph of $G$ with vertex set $V(G)\setminus V(C)$. Thus $\Delta(G\setminus C)=\Delta(G[V(G)\setminus V(C)])$, is maximum degree of the induced subgraph of $G$ outside $C$ and $|G\setminus C|=|G|-|C|$ is the number of vertices $G$ has outside $C$.

To the author's knowledge, nothing like (\ref{MainQuestion}) has been considered before. 
Why would anyone care about such things?
Perhaps the main reason is that the theorems we prove look very natural. 
Therefore,  results like (\ref{MainQuestion}) seem to be an intriguing new direction for the neverending quest of understanding cycles in graphs.
However, if people want applications, we also got applications. 
In Section~\ref{SectionConclusion} we will use something like~(\ref{MainQuestion}) to make progress on a conjecture of  Erd\H{o}s, Gy\'arf\'as, and Pyber. Our original motivation for studying~(\ref{MainQuestion}) came from applying it to following conjecture of Lehel. Recall that for a graph $G$, the complement  of $G$, denoted $\overline G$, is the graph on $V(G)$ whose edges are exactly the nonedges of $G$.
\begin{conjecture}[Lehel]
The vertices of every graph $G$ can be covered by a cycle $C_1$ in $G$ and a vertex-disjoint cycle $C_2$ in the complement of $G$.
\end{conjecture}
In this conjecture, a single edge, a single vertex, and the empty set are all considered to be cycles. This is to avoid some trivial counterexamples.  This convention will be used for the rest of this paper.

This conjecture attracted a lot of attention in the '90s and early '00s. The conjecture first appeared in Ayel's PhD thesis~\cite{Ayel} where it was proved for some special types of colourings of $K_n$.  Gerencs\'er and Gy\'arf\'as~\cite{Gerencser} showed that the  conjecture is true if $C_1$ and $C_2$ are required to be paths rather than cycles.
Gy\'arf\'as~\cite{Gyarfas2} showed that the conjecture is true if $C_1$ and $C_2$ are allowed to intersect in one vertex.
\L{uczak}, R\"odl, and Szemer\'edi~\cite{Szemeredi2} showed that the conjecture holds for sufficiently large graphs.   Later, Allen~\cite{Allen} gave an alternative proof that works for smaller (but still large) graphs.
Lehel's Conjecture was finally shown to be true for all graphs by Bessy and Thomass\'e~\cite{Thomasse}.
\begin{theorem}[Bessy and Thomass\'e, \cite{Thomasse}]\label{TheoremBessyThomasse}
The vertices of every graph $G$ can be covered by a cycle $C_1$ in $G$ and a vertex-disjoint cycle $C_2$ in the complement of $G$.
\end{theorem}

How is Lehel's Conjecture related to~(\ref{MainQuestion})? If we could prove that ``every graph $G$ has a cycle with $\Delta(G\setminus C)\leq \frac{1}{2}|G\setminus C|-1$,'' then we would have a strengthening of Lehel's Conjecture. To see this notice that for a graph $H$, $\Delta(H)\leq \frac{1}{2}|H|-1$ is equivalent to $\delta(\overline H)\geq \frac{1}{2}|\overline H|$. Thus Dirac's Theorem implies that for any $H$ with $\Delta(H)\leq \frac{1}{2}|H|-1$, the complement of $H$ is Hamiltonian. So if it \emph{were} true that every graph $G$ had a cycle  with $\Delta(G\setminus C)\leq \frac{1}{2}|G\setminus C|-1$, then by Dirac's Theorem, we would have a cycle in $\overline G$ covering $G\setminus C$, and so an alternative proof of Lehel's Conjecture.

\begin{figure}
  \centering
    \includegraphics[width=0.40\textwidth]{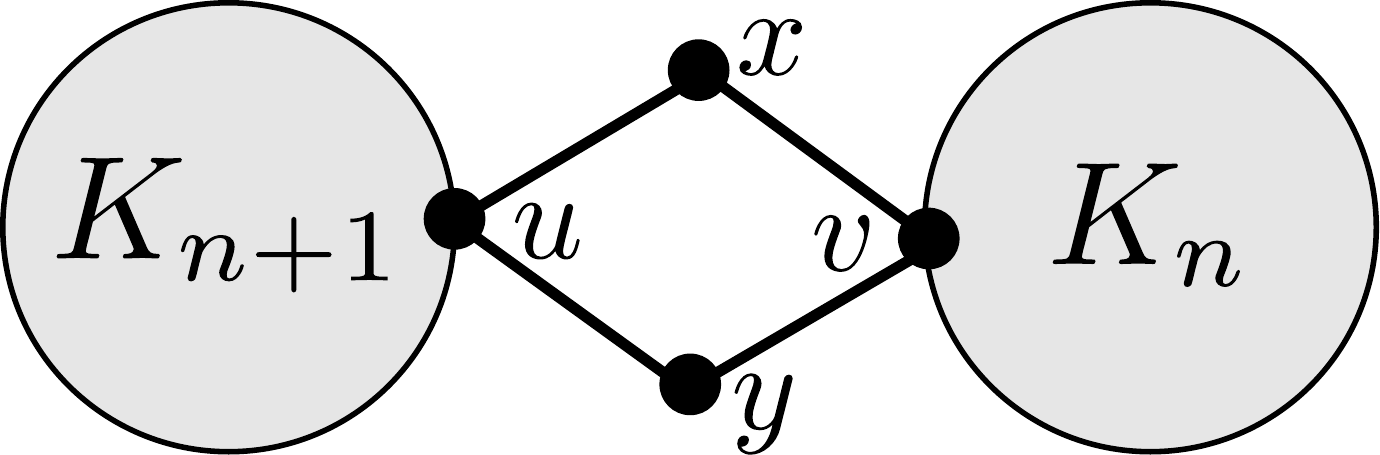}
  \caption{A graph with $2n+3$ vertices in which every cycle has $\Delta(G\setminus C)\geq \frac{1}{2}(|G\setminus C|-1)$.  To see this note that if $C$ does not pass through $K_{n+1}$, we have  $u\in A$ and $d_{G\setminus C}(u)\geq\frac{1}{2}|G\setminus C|$.  If $C$ passes through $K_{n+1}$ and $v\not\in C$ then we have $d_{G\setminus C}(v)\geq\frac{1}{2}|{G\setminus C}|$.  The only remaining case is that $C$ is the 4-cycle passing through both $u$ and $v$, in which case any vertex in $K_{n+1}-u$ will have degree $\frac{1}{2}(|G\setminus C|+1)$ in $G\setminus C$.}\label{FigureGeneralGraphExtremal}
\end{figure}

Unfortunately this strengthening of Lehel's Conjecture simply isn't true. See Figure~\ref{FigureGeneralGraphExtremal} for a graph $G$ which has  $\Delta(G\setminus C)\geq \frac{1}{2}(|G\setminus C|-1)$ for every cycle $C$. But if the maximum degree condition is weakened slightly, then we can prove it.
\begin{theorem} \label{CycleMinDegree}
Every graph $G$ has a cycle $C$ with wither $|G\setminus C|=0$ or
\begin{equation}\label{EqCycleMinDegree}
\Delta(G\setminus C)\leq \frac{1}{2}(|G\setminus C|-1).
\end{equation}
\end{theorem}
The above theorem isn't quite strong enough to directly imply Lehel's Conjecture. However, with more work, it is possible to use it to give an alternative proof of the conjecture. This shouldn't be very surprising since using known results about Hamiltonicity, it is easy to show that graphs $H$ satisfying (\ref{EqCycleMinDegree}) whose complements aren't Hamiltonian must have a very limited structure. We'll briefly discuss how to use Theorem~\ref{CycleMinDegree} to get an alternative proof of the Bessy-Thomass\'e Theorem in Section~\ref{SectionConclusion}.

Theorem~\ref{CycleMinDegree} is proved in Section~\ref{SectionCycleMinDegreeDisconnected}. In Section~\ref{SectionCycleMinDegreeSimple}  a slight weakening of Theorem~\ref{SectionCycleMinDegreeSimple} is proved with (\ref{EqCycleMinDegree}) replaced by ``$\Delta(G\setminus C)\leq \frac{1}{2}|G\setminus C|$''.  This weakening has a much easier proof than Theorem~\ref{CycleMinDegree} and serves to illustrate the ideas we use in the main theorem.

Theorem~\ref{CycleMinDegree} is a very natural result of the form (\ref{MainQuestion}) which we set out to investigate. Our theorem shows that results like~(\ref{MainQuestion}) do hold, and finds the best possible such result which holds for general graphs.
Indeed the graph $G$ in Figure~\ref{FigureGeneralGraphExtremal} shows that the degree condition (\ref{EqCycleMinDegree}) cannot be decreased in general. 
 There are various possible directions for further research. One direction is to change the degree condition on $H$ from maximum degree to something else. For example is it true that every graph $G$ can be partitioned into a cycle and an induced subgraph $H$ such that $\overline{H}$ satisfies Ore's condition for Hamiltonicity? We discuss this and similar open problems in Section~\ref{SectionConclusion}.

Another direction for further research is to see if Theorem~\ref{CycleMinDegree} can be improved if we impose some extra conditions on the graph $G$. The graphs $G$ in Figure~\ref{FigureGeneralGraphExtremal} are the only ones we currently know which don't have a cycle $C$ with $\Delta(G[V(G)\setminus C])\leq \frac{1}{2}|V(G)\setminus C|-1$. It seems likely that if mild conditions are imposed on the graph $G$ in Theorem~\ref{CycleMinDegree}, then the degree condition (\ref{EqCycleMinDegree}) could be improved. The second main result of this paper is to improve Theorem~\ref{CycleMinDegree} in the case when $G$ is highly connected.
\begin{theorem} \label{CycleMinDegreeConn} 
Every $k$-connected graph $G$ has a cycle $C$ with 
\begin{equation}\label{EqCycleMinDegreeConn}
\Delta(G\setminus C)\leq \frac{1}{k+1}|G\setminus C|+3.
\end{equation}
\end{theorem}
Comparing the bound (\ref{EqCycleMinDegreeConn}) to (\ref{EqCycleMinDegree}), we see that for $k\geq 2$ (\ref{EqCycleMinDegreeConn}) is better unless $|H|$ is very small.
There is a sense in which Theorem~\ref{CycleMinDegreeConn} is tight---For every $k\geq 1$, there are $k$-connected graphs $G$ which don't have a cycle $C$ with $\Delta(G[V(G)\setminus C])\leq \frac{1}{k+1}|V(G)\setminus C|-\frac{2k+1}{k+1}.$ See Figure~\ref{FigureKConnectedExtremal} for an example of such graphs. This shows that the ``$\frac{1}{k+1}|H|$'' term in (\ref{EqCycleMinDegreeConn}) cannot be decreased. However the ``$+3$'' term is not optimal and can significantly probably be improved for all $k$.

Theorem~\ref{CycleMinDegreeConn} is proved in Section~\ref{SectionConnected}. We remark that our proof actually gives a slightly better bound ``$\Delta(G\setminus C)\leq \frac{1}{k+1}|G\setminus C|+3-\frac{4}{k+1}$'' instead of (\ref{EqCycleMinDegreeConn}). We don't place much emphasis on this since neither bound is tight and (\ref{EqCycleMinDegreeConn})  looks cleaner.

\begin{figure}
  \centering
    \includegraphics[width=0.55\textwidth]{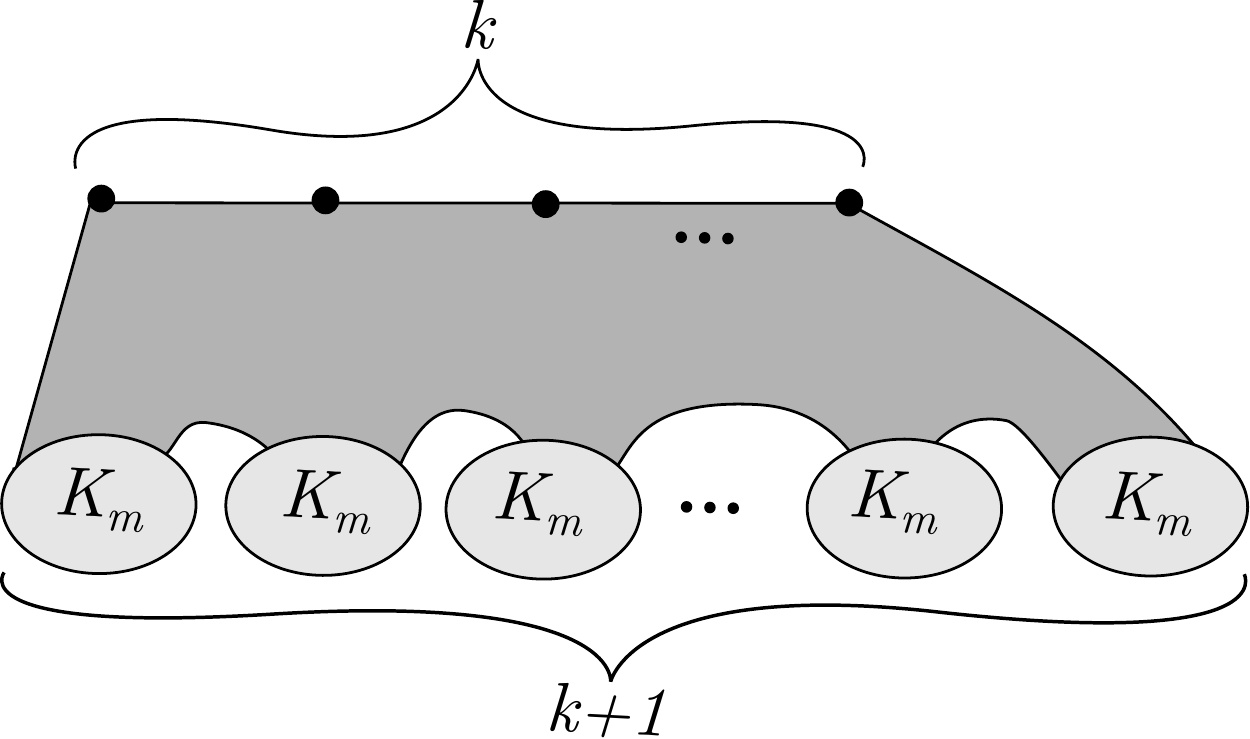}
  \caption{A $k$-connected graph in which everty cycle has $\Delta(G[V(G)\setminus C])\geq \frac{1}{k+1}|V(G)\setminus C|--\frac{2k+1}{k+1}$. The graph consists of $k+1$ copies of $K_m$, and $k$ extra vertices. All the edges are present except those going between the $K_m$s. To see that every cycle in this graph has $\Delta(G[V(G)\setminus C])\geq \frac{1}{k+1}|V(G)\setminus C|--\frac{2k+1}{k+1}$, notice that a cycle cannot intersect all the copies of $K_m$ (since it must use at least one of the extra vetices every time it moves from one $K_m$ to the next.) Therefore we have $\Delta(G[V(G)\setminus C])\geq \Delta(K_m)=m-1=\frac{1}{k+1}|V(G)|-\frac{k}{k+1}-1\geq \frac{1}{k+1}|V(G)\setminus C|-\frac{2k+1}{k+1}$ as required. 
  }\label{FigureKConnectedExtremal}
\end{figure}

We prove Theorem~\ref{CycleMinDegreeConn} with a particular application in mind. It can be used to prove a new approximate version of a conjecture due to Erd\H{o}s, Gy\'arf\'as, and Pyber. We describe this application in the next section.

\subsubsection*{An application: Partitioning a 3-coloured $K_n$ into monochromatic cycles.}
Erd\H{o}s, Gy\'arf\'as \& Pyber conjectured the following.
\begin{conjecture} [Erd\H{o}s, Gy\'arf\'as \& Pyber, \cite{Erdos}] \label{Erdos}
The vertices of every $r$-edge-coloured complete graph can be covered with $r$ vertex-disjoint monochromatic cycles.
\end{conjecture}
Notice that the $r=2$ case of the above conjecture follows from Lehel's Conjecture. Indeed given a $2$-edge-coloured $K_n$, applying the Bessy-Thomass\'e Theorem to the red colour class of $K_n$, we obtain a covering of $K_n$ by a red cycle and a disjoint blue cycle.

The Erd\H{o}s-Gy\'arf\'as-Pyber Conjecture has attracted a lot of attention. See~\cite{GyarfasSurvey} for a detailed survey of the work related to this conjecture. Despite many positive results~(see \cite{GyarfasSurvey}), the conjecture turned out to be false for all $r\geq 3$~\cite{PokrovskiyCycles}. But not all hope is lost---the counterexamples constructed in~\cite{PokrovskiyCycles} are only barely counterexamples---in all the $r$-edge-coloured complete graphs constructed in~\cite{PokrovskiyCycles}, it is possible to find $r$ disjoint monochromatic cycles covering all, except one, vertex in the graph. Therefore it is still possible that approximate versions of the conjecture are true. 
For $r=3$, Gy\'arf\'as, Ruszink\'o, S\'ark\"ozy, and Szemer\'edi proved the following approximate version of Conjecture~\ref{Erdos}.
\begin{theorem} [Gy\'arf\'as, Ruszink\'o, S\'ark\"ozy \& Szemer\'edi, ~\cite{Szemeredi1}] \label{threecyclesSzemeredi}
Every $3$-edge-coloured $K_n$ has $3$ vertex-disjoint monochromatic cycles covering $n-o(n)$ vertices.
\end{theorem}
The proof of this theorem uses the Regularity Lemma, and so the bound on $o(n)$ in Theorem~\ref{threecyclesSzemeredi} is not very good.  In particular the function $o(n)$ in Theorem~\ref{threecyclesSzemeredi} tends to infinity with $n$. 
In this paper we use Theorem~\ref{CycleMinDegreeConn} to prove an improved approximate version of the $r=3$ case of Conjecture~\ref{Erdos} where only a constant number of vertices are left uncovered.
\begin{theorem}  \label{ThreeCycles}
For sufficiently large $n$, every $3$-edge-coloured $K_n$ has $3$ vertex-disjoint monochromatic cycles covering $n-\cCycles$ vertices.
\end{theorem}
This theorem is proved in Section~\ref{Section3Cycles}.
The same theorem was discovered independently by Letzter~\cite{Letzter3Cycles} who proved it using completely different methods, and with a better constant of $60$ rather than $\cCycles$.

\subsection*{Notation}
Throughout this paper we will use additive notation to add/subtract vertices i.e. if $G$ is a graph and $S\subseteq V(G), v\in V(G)$, then ``$S+v$'' means $S\cup\{v\}$ and $S-v$ means $S\setminus \{v\}$.
We will also use additive notation for concatenating paths---for two paths $P=p_1 \dots p_i$ and $Q=q_1 \dots q_j$, $P + Q$ denotes the path with vertex sequence $p_1\dots p_i q_1\dots q_j$.


\section{A short, simple theorem}\label{SectionCycleMinDegreeSimple}
This section is purely for exposition. 
Here we give a slight weakening of Theorem~\ref{CycleMinDegree}, which has a much shorter proof.
\begin{theorem} \label{CycleMinDegreeSimple}
Every graph $G$ has a cycle $C$ with
\begin{equation}\label{EqCycleMinDegreeSimple}
\Delta(G\setminus C)\leq \frac{1}{2}|G\setminus C|.
\end{equation}
\end{theorem}
The proof of this theorem illustrates the main idea of the proof of Theorem~\ref{CycleMinDegree}.
The key idea is to consider the following lemma which, unlike Theorem~\ref{CycleMinDegreeSimple}, can be proved by induction.
\begin{lemma} \label{PathMinDegreeSimple}
Let $G$ be a graph, and  $A$ and $B$ be sets of vertices in $G$ such that $|A|\geq \frac{1}{2}|G|$ and $|B|> \frac{1}{2}|G|$ both hold. There there is a path $P$ between a vertex in $A$ and a vertex in $B$ such that we have 
\begin{equation}\label{PathDegreeSimple}
\Delta(G\setminus P) \leq \frac{1}{2}|G\setminus P|.
\end{equation}
\end{lemma}
\begin{proof}
The proof is by induction.   The lemma holds trivially when $|G|\leq 2$ (taking $P$ to be any vertex in $A\cap B$ which is non-empty since $|A|+|B|>|G|$).  Suppose that the lemma holds for all graphs of order less than $n$. Let $G$ be a graph of order $n$.

If $\Delta(G)\leq\frac{1}{2}(|G|-1)$, then the lemma holds by choosing $P$ to be any vertex in $A\cap B$.
Therefore, we can suppose that there is a vertex $v\in G$ such that $d(v)\geq\frac{1}{2}|G|$.

We define a subgraph $G'$ of $G$, and two sets  $A'$ and $B'\subseteq V(G')$ as follows.
\begin{enumerate}[\normalfont(i)]
 \item If $v\in A$, let $G'=G-v$, $A'=B-v$, and $B'=N(v)$.
 \item If $v\not\in A$ and $v\in B$, let $G'=G-v$, $A'=A$, and $B'=N(v)$.
 \item If $v\not\in A, B$, then notice that $|N(v)+v|+|A|\geq |G|+1$ implies that $N(v)\cap A$ contains some vertex $u$.  Let $G'=G-v-u$, $A'=N(v)-u$, and $B'=B-u$.
\end{enumerate}
Notice that in all three cases, we have $|A'|\geq \frac{1}{2}|G'|$ and $|B'|> \frac{1}{2}|G'|$.  By induction, there there is a path $P'$ in $G'$ starting in $A'$ and ending in $B'$ such that we have $\Delta(G'\setminus P') \leq \frac{1}{2}|G'\setminus P'|$.  
If we are in case (i) or  (ii) the the path $P'+v$  satisfies the lemma.
If we are in case  (iii) the the path $u+v+P'$  satisfies the lemma.
\end{proof}

We now prove the main result of this section.
\begin{proof}[Proof of Theorem~\ref{CycleMinDegreeSimple}]
If $\Delta(G)\leq \frac{1}{2}|G|$, then this holds by taking $C=\emptyset$.  Otherwise there is a vertex $v$ of degree greater than $\frac{1}{2}|G|$ in $G$.  We can apply Lemma~\ref{PathMinDegreeSimple} to  $G-v$ with $A=B=N(v)$. This gives us a path $P$ from $N(v)$ to $N(v)$ such that $P+v$ is a cycle and $\Delta(G\setminus (P+v)) \leq \frac{1}{2}|G\setminus (P+v)|$.  
We can close the path using the vertex $v$ to obtain a cycle $C$ satisfying $\Delta(G\setminus C) \leq \frac{1}{2}|G\setminus C|$. 
\end{proof}

The full proof of Theorem~\ref{CycleMinDegree} is similar since it also uses a version of Lemma~\ref{PathMinDegreeSimple}.  The difficulty in proving Theorem~\ref{CycleMinDegree} comes from the fact that Lemma~\ref{PathMinDegreeSimple} does not hold with (\ref{pathdegree}) replaced by $\Delta(G\setminus P) \leq \frac{1}{2}\left(|G\setminus P|-1\right)$ 
To see this, consider a graph $G$  consisting of a copy of $K_m$ and a disjoint $K_{m+1}^-$ and no other edges (here, $K_{m+1}^-$ denoted $K_{m+1}$ with an edge.) Let $A=V(K_m)+u$  and $B=V(K_m)+v$ where $uv$ is the edge which was deleted from $K_{m+1}$. It is easy to check that $\Delta(G\setminus P) \geq \frac{1}{2}|G\setminus P|$ holds for any path $P$ going from $A$ to $B$.

Lemma~\ref{PathMinDegreeSimple} being best possible is the barrier to generalizing the above proof to get Theorem~\ref{CycleMinDegree}. 
This barrier is overcome in the next section by classifying the extremal graphs for Lemma~\ref{CycleMinDegreeSimple}.

\section{Partitioning general graphs}\label{SectionCycleMinDegreeDisconnected}
In this section we prove Theorem~\ref{CycleMinDegree}.
It is recommended that the reader familiarise themselves with the proof of Theorem~\ref{CycleMinDegreeSimple} before reading the proof in this section.

The main step in the proof of Theorem~\ref{CycleMinDegree} will be to prove a version of Lemma~\ref{PathMinDegreeSimple} with the identity (\ref{PathDegreeSimple}) replaced by  $\Delta(G\setminus P) \leq \frac{1}{2}\left(|G\setminus P|-1\right)$.
Before we can state the lemma which we will prove we will need some notation.

\begin{definition}\label{DefinitionBalancedComponents}
We say that a graph $G$ has {\bf balanced components} if $V(G)$ can be partitioned into two sets $X$ and $Y$ such that $X$ and $Y$ have no edges between them and $|X|=\left\lfloor {|G|}/{2} \right\rfloor$, $|Y|=\left\lceil {|G|}/{2} \right\rceil$.
\end{definition}

A graph with no vertices is denoted by $\emptyset$. 
For convenience, throughout this section we set $\Delta(\emptyset)=\delta(\emptyset)=-\frac{1}{2}$.  
Though a bit unintuitive, this notation is actually extremely natural---it is the only assignment of values to $\Delta(\emptyset)$ and $\delta(\emptyset)$ which ensures that we have $\overline \emptyset=\emptyset$,  $\delta(G)+\Delta(\overline{G})=|G|-1$, and $\Delta(G)\geq \delta (G)$ for all graphs $G$.
For our purposes, it also it also ensures that we have $\Delta(\emptyset)\leq \frac{1}{2}(|\emptyset|-1)$.
Therefore, with this convention, the possibility of $C$ being empty does not need to be stated explicitly in the statement of Theorem~\ref{CycleMinDegree}.  Also, some of our proofs are simplified with this notation, since without it the case when $G\setminus C=\emptyset$ would have to be dealt with separately from the main proof.

The following lemma is the crucial step of the proof of Theorem~\ref{CycleMinDegree}.
\begin{lemma} \label{mainlemma}
Let $G$ be a graph on at least two vertices which does not have balanced components.  Let $A$ and $B$ be sets of vertices in $G$ such that $|A|\geq \frac{1}{2}(|G|-1)$ and $|B|\geq \frac{1}{2}|G|$ both hold. Then there is a path $P$ starting in $A$ and ending in $B$ such that we have 
\begin{equation}\label{pathdegree}
\Delta(G\setminus P) \leq \frac{1}{2}(|G\setminus P|-1).
\end{equation}
\end{lemma}

The following lemma gives some simple examples of graphs $H$ with $\Delta(H)\leq \frac{1}2(|H|-1)$.
\begin{lemma}\label{LemmaLowDegreeDegreeStructures}
Suppose that $H$ is a graph with one of the following structures.
\begin{enumerate}[(i)]
\item $H$ has balanced components.
\item $V(H)$ has a partition into two sets $X$ and $Y$ such that there are no edges between $X$ and $Y$,  $|X|\leq |Y|\leq |X|+3$, and $\Delta(Y)<|Y|-1$.
\item $|H|$ is odd, and $H$ has a vertex $v$ with $d(v)\leq(|H|-1)/2$ such that $H-v$ has balanced components.
\item  $V(H)$ has a partition into nonempty $X, Y_1, Y_2,$ and $\{v\}$ such that there are no edges between any of $X, Y_1$, and $Y_2$, $N(v)\subseteq X\cup Y_1$, $|X|+2=|Y_1|+|Y+2|$, and $d(v)\leq(|H|-1)/2$.
\end{enumerate}
Then $\Delta(H)\leq \frac{1}2(|H|-1)$.
\end{lemma}
\begin{proof}
\begin{enumerate}[(i)]
\item Let $V(H)$ have a partition into balanced components $X$ and $Y$ as in Defintion~\ref{DefinitionBalancedComponents}. Vertices in $X$ have degree $\leq |X|-1\leq |H|/2-1$. Vertices in $Y$ have degree $\leq |Y|-1\leq (|H|+1)/2-1$.
\item Vertices in $X$ have degree $\leq |X|-1\leq |H|/2-1$. Vertices in $Y$ have degree $\leq |Y|-2\leq |Y|/2+(|X|+3)/2-2=(|H|-1)/2$.
\item Let $V(H)-v$ have a partition into balanced components $X$ and $Y$ as in Defintion~\ref{DefinitionBalancedComponents}. Since $|H|$ is odd, we have $|X|=|Y|=(|H|-1)/2$
Vertices  in $X$ have degree $\leq |X|\leq (|H|-1)/2$. Similarly vertices  in $Y$ have degree $\leq |Y|\leq (|H|-1)/2$.
\item Vertices in $X$ have degree $\leq |X|= (|H|-3)/2$. Since $Y_2$ is nonempty, vertices in $Y_1$ have degree $\leq |Y_1|\leq |X|+1\leq (|H|-1)/2$. Since $Y_1$ is nonempty, vertices in $Y_2$ have degree $\leq |Y_2|-1\leq |X|\leq (|H|-3)/2$. 
\end{enumerate}
\end{proof}

The proof of Lemma~\ref{mainlemma} is by an induction similar to the proof of Lemma~\ref{PathMinDegreeSimple}.  However the condition of ``$G$ does not have balanced components'' in the statement of Lemma~\ref{mainlemma} makes the proof substantially more involved.   This is because proving the ``initial case'' of the induction now requires verifying Lemma~\ref{mainlemma} whenever $G$ is in some sense ``close to having balanced components''.
This is performed in the following two lemmas.

\begin{figure}
  \centering
    \includegraphics[width=0.65\textwidth]{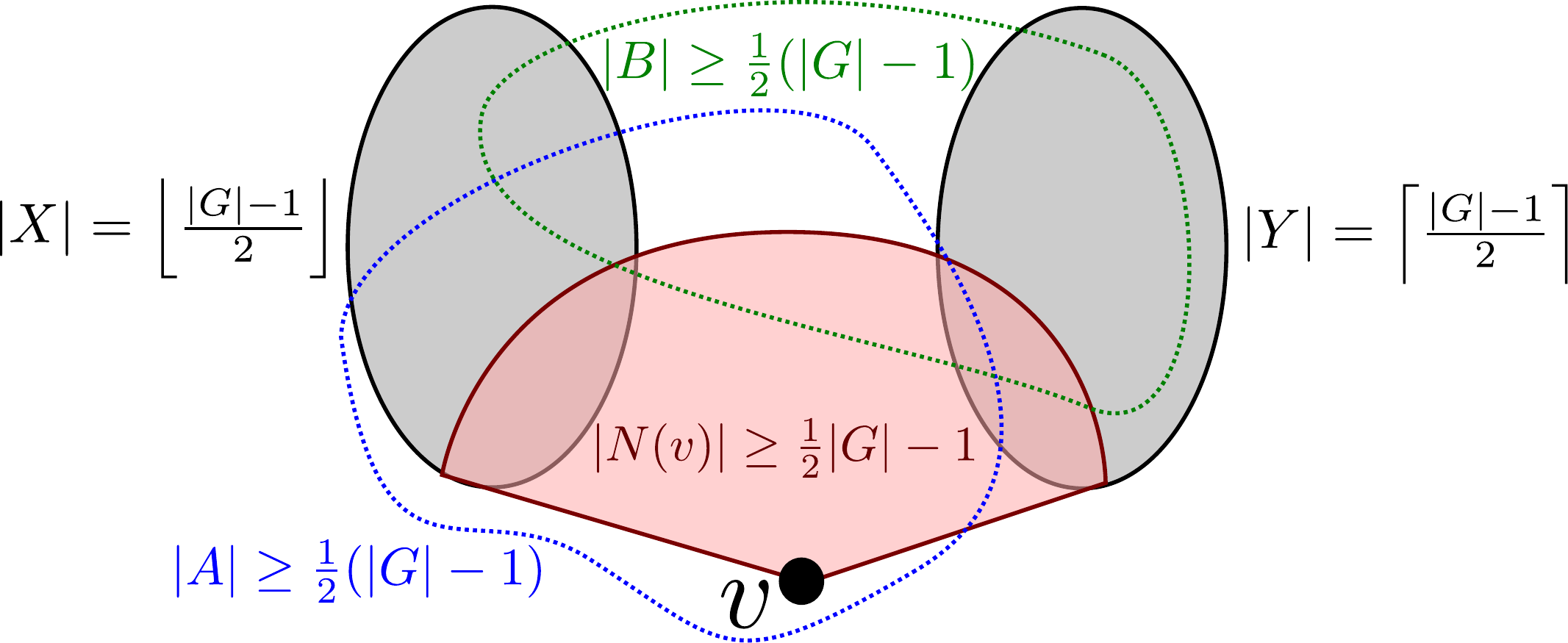}
  \caption{A diagram of the set up for Lemma~\ref{1vertex}.
  }\label{FigureLemma1Vertex}
\end{figure}
\begin{lemma}\label{1vertex}
Let $G$ be a graph of order at least $4$ which does not have balanced components.  Suppose that $G$ has a vertex $v$ such that $G\setminus\{v\}$ has balanced components and $d(v)\geq\frac{1}{2}|G|-1$.  Let $A$, $B\subseteq V(G)$ such that $|A|\geq \frac{1}{2}(|G|-1)$, $|B|\geq \frac{1}{2}(|G|-1)$ and $v \in A $ all hold.  Then there is a path $P$ starting in $A$ and ending in $B$ such that we have 
\begin{equation}\label{pathdegree1vertex}
\Delta(G\setminus P) \leq \frac{1}{2}(|G\setminus P|-1).
\end{equation} 
\end{lemma}
\begin{proof}
See Figure~\ref{FigureLemma1Vertex} for a diagram of the graph $G$ and the sets $A$ and $B$ for this lemma.
Let $X$ and $Y$ be two sets partitioning $G\setminus\{v\}$ as in the definition of having balanced components.  Note that if $v \in B$ then the lemma follows by taking $P=\{v\}$, since then $G\setminus P$ has balanced components and so satisfies (\ref{pathdegree1vertex}). Therefore, we can suppose for the remainder of of proof that  we have $v \not\in B$.
We split into two cases depending on whether $|G|$ is odd or even.

Suppose that $|G|$ is odd.   In this case we have $|X|=|Y|=\frac{1}{2}(|G|-1)$.  Since $G$ does not have balanced components, $v$ must have neighbours in both $X$ and $Y$. We consider two cases depending on whether $B\cap N(v)$ is empty or not.

{\bf Case 1:}  Suppose that we have $B\cap N(v)\neq \emptyset$.  Let $v_b$ be a vertex in $B\cap N(v)$.  We can let $P= v v_b$ in order to obtain a path such that $G\setminus P$ has balanced components and so satisfies~(\ref{pathdegree1vertex}).

{\bf Case 2:}  Suppose that we have $B\cap N(v) = \emptyset$.
In this case $v\neq B$ gives $|B|+|N(v)+v|\geq |G|-\frac{1}{2}$. This, together with the integrality of $|B|$ and $|N(v)+v|$ implies that we must have $|B|=\frac{1}{2}(|G|-1)$ and $|N(v) + v|=\frac{1}{2}(|G|+1)$.  Thus $B$ and $N(v)+v$ partition $G$.  The fact that $N(v)\cap X\neq \emptyset$, and $N(v)\cap Y\neq \emptyset$ implies that $B\cap X \neq \emptyset$ and $B\cap Y \neq \emptyset$.  The assumption that $G$ does not have a partition into balanced components $B$ and $N(v)+v$  implies that there is an edge between $N(v)$ and either $B\cap X$ or $B\cap Y$.  Notice that so far, the roles of $X$ and $Y$ in the proof of this case have been identical.  Therefore, without loss of generality, we may assume that there is an edge between some vertices $u_x\in N(v)\cap X$ and  $b_x\in B\cap X$. We split into two cases depending on the values of $\Delta(X)$ and $\Delta(Y)$.

{\bf Case 2.1:} Suppose that we have either $\Delta(X)<|X|-1$ or $\Delta(Y)<|Y|-1$.
 If $\Delta(Y)<|Y|-1$ holds, then $P=v u_x b_x$ is a path with $G\setminus P$  having structure (ii) of Lemma~\ref{LemmaLowDegreeDegreeStructures} and so satisfies (\ref{pathdegree2vertices}).
 
If $\Delta(Y)=|Y|-1$ then $Y$ is connected and so there must be an edge between some vertices $b_y \in B\cap Y$ and $u_y \in N(v)\cap Y$. In this case, by the assumption of Case 2.1, we must also  have  $\Delta(X)<|X|-1$.  Therefore  $P=v u_y b_y$ is a path with $G\setminus P$  having structure (ii) of Lemma~\ref{LemmaLowDegreeDegreeStructures} and so satisfies (\ref{pathdegree2vertices}).

{\bf Case 2.2:}  Suppose that we have both $\Delta(X)=|X|-1$ and $\Delta(Y)=|Y|-1$.
  Then $Y$ must be connected, and so there is an edge between some vertices $u_y\in N(v)\cap Y$ and  $b_y\in B\cap Y$.
  
  Since $|G|\geq 4$, $A$ has at least two vertices, and so either  $A\cap X$ or $A \cap Y$ is not empty.  Without loss of generality, we may assume that $A\cap X$ contains a vertex $a_x$.  Let $x$ be a vertex of degree $|X|-1$ in $G[X]$ and $v_x$ a vertex in $N(v)\cap X$.    Notice that, depending on which of $a_x$, $x$ and $v_x$ are equal, one of the sequences $a_x x v_x v u_y b_y$, $a_x v_x v u_y b_y$ or $a_x v u_y b_y$ forms a path from $A$ to $B$.  Choose $P$ to be this path in order to obtain a path such that  $G\setminus P$ has balanced components and so satisfies~(\ref{pathdegree1vertex}).

\

Suppose that $|G|$ is even.   In this case we have $|X|=\frac{1}{2}|G|-1$ and $|Y|=\frac{1}{2}|G|$. Notice that since $G$ does not have balanced components, $N(v)\cap Y$ contains some vertex, $v_y$.  We consider three cases, depending on which of $N(v)\cap B \cap Y$ or $N(v)\cap B \cap X$ are empty.

{\bf Case 1:}  Suppose that we have $N(v)\cap B\cap Y \neq \emptyset$.  Let $v_b$ be a vertex in $B\cap N(v) \cap Y$.  In this case we can let $P= v v_b$ to obtain a path such that $G\setminus P$ has balanced components and so satisfies (\ref{pathdegree1vertex}).

{\bf Case 2:}  Suppose that we have $N(v)\cap B\cap X \neq \emptyset$.  Let $v_b$ be a vertex in $B\cap N(v) \cap X$.
There are two subcases depending on the whether we have $\Delta(Y)=|Y|-1$ or not.

{\bf Case 2.1:}  Suppose that we have $\Delta(Y)< |Y|-1$.  In this case  $P= v v_b$ is a path with $G\setminus P$  having structure (ii) of Lemma~\ref{LemmaLowDegreeDegreeStructures} and so satisfies (\ref{pathdegree2vertices}).

{\bf Case 2.2:}  Suppose that we have $\Delta(Y)= |Y|-1$.  Let $y$ be a vertex of degree $|Y|-1$ in $G[Y]$.  There are two subcases depending on whether $A\cap Y$ is empty or not.

{\bf Case 2.2.1:} 
Suppose that we have $A\cap Y \neq \emptyset$. Let $a_y$ be a vertex in $A\cap Y$.  Depending on which of $a_y$, $v_y$ and $y$ are equal, one of the sequences $a_y y v_y v v_b$, $a_y v_y v v_b$ or $a_y v v_b$ forms a path from $A$ to $B$.  Choose $P$ to be this path to obtain a path such that  $G\setminus P$ has balanced components and so satisfies (\ref{pathdegree1vertex}).

{\bf Case 2.2.2:} 
Suppose that we have $A\cap Y = \emptyset$.  This implies that $A= X + v$ must hold, so we have $v_b \in A$.  Note that $|B|+ |Y|\geq |G|-\frac{1}{2}$, together with the fact that $v$ is in neither $B$ nor $Y$ means that $B\cap Y$ contains a vertex $b_y$.  Depending on which of $b_y$, $v_y$ and $y$ are equal, one of the sequences $v_b v v_y y b_y$, $v_b v_y y b_y$ or $v_b y b_y$ forms a path from $A$ to $B$.  Choose $P$ to be this path to obtain a path such that  $G\setminus P$ has balanced components and so satisfies (\ref{pathdegree1vertex}).

{\bf Case 3:}  Suppose that we have $N(v)\cap B= \emptyset$.  Note that $|N(v)+v|+|B|\geq |G|-\frac{1}{2}$, together with the integrality of $|N(v)+v|$ and $|B|$ implies that we have $|N(v)+v|=|B|=\frac{1}{2}|G|$. There are two subcases depending on thether there are edges between $B\cap Y$ and $N(v)$ or not.

{\bf Case 3.1:}
If there is an edge $b_y v_y$ between $B\cap Y$ and $N(v)$ then $P=v v_y b_y$ is a path such that $G\setminus P$ has balanced components and so satisfies (\ref{pathdegree1vertex}).

{\bf Case 3.2:}
If there are no edges between $B\cap  Y$ and $N(v)$, then note that since $G$ does not have a partition into balanced components $N(v)+v$ and $B$, there must be an edge $b_x v_x$ between $B\cap X$ and $N(v)$.  
Since $|B|=|Y|=\frac{1}{2}|G|$ and $v \not\in B\cup Y$, we have that $B\cap Y$ is not empty.
Since we have $N(v)\cap Y\neq \emptyset$, $B\cap Y \neq \emptyset$, and there are no edges between $B\cap  Y$ and $N(v)$, we obtain $\Delta(Y)< |Y|-1$.  This implies that  $P = v v_x b_x$ is a path with $G\setminus P$  having structure (ii) of Lemma~\ref{LemmaLowDegreeDegreeStructures} and so satisfies (\ref{pathdegree2vertices}).
\end{proof}

\begin{figure}
  \centering
    \includegraphics[width=0.65\textwidth]{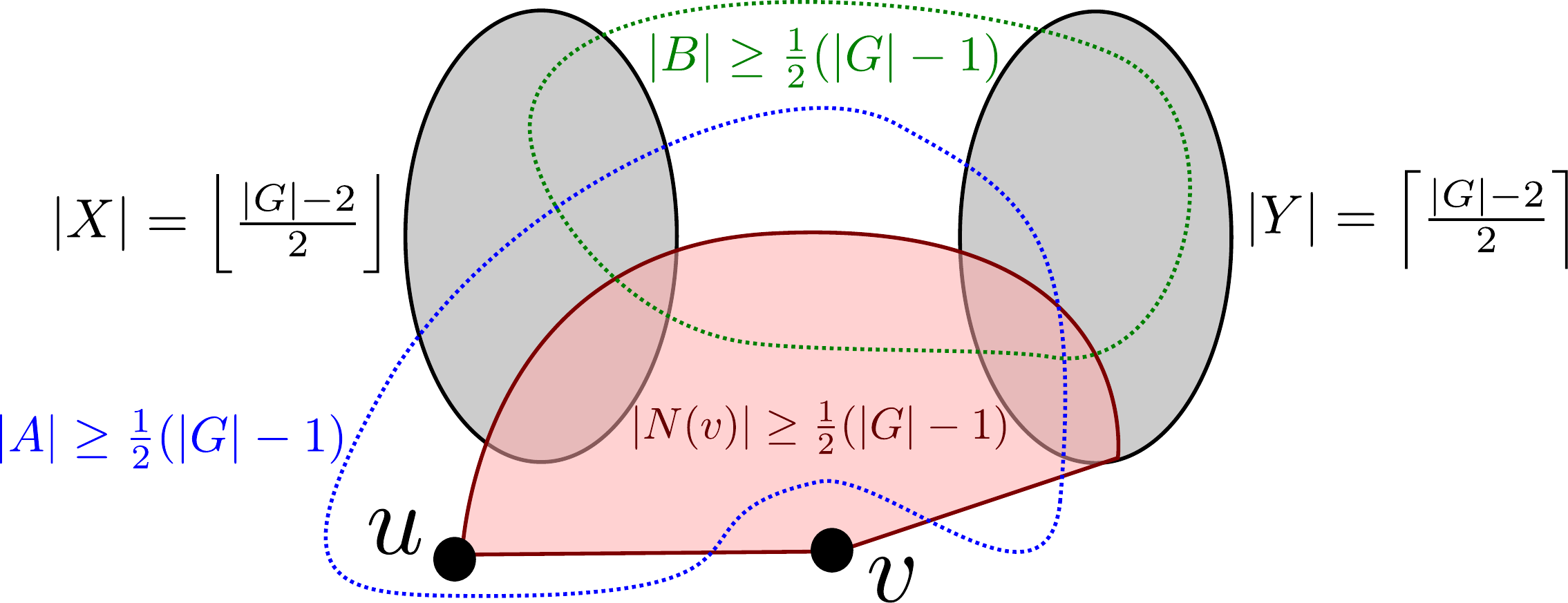}
  \caption{A diagram of the set up for Lemma~\ref{2vertices}.
  }\label{FigureLemma2Vertex}
\end{figure}
The following lemma is similar to Lemma~\ref{1vertex}, except there are now two vertices $u$ and $v$ playing the role $v$ played in Lemma~\ref{1vertex}. 
\begin{lemma}\label{2vertices}
Let $G$ be a graph on at least $4$ vertices which does not have balanced components.  Suppose that  that $G$ has two vertices $u$ and $v$ such that $G\setminus\{u,v\}$ has balanced components, $uv$ is an edge and $d(v)\geq\frac{1}{2}(|G|-1)$.  Let $A$, $B\subseteq V(G)$ such that $|A|, |B|\geq \frac{1}{2}(|G|-1)$, $\max(|A|, |B|)\geq \frac{1}{2}|G|$, $u \in A$ and $v \not\in A\cup B$ all hold.  Then there is a path $P$ starting in $A$ and ending in $B$ such that we have 
\begin{equation}\label{pathdegree2vertices}
\Delta(G\setminus P) \leq \frac{1}{2}(|G\setminus P|-1).
\end{equation} 
\end{lemma}

\begin{proof}
See Figure~\ref{FigureLemma2Vertex} for a diagram of the graph $G$ and the sets $A$ and $B$ for this lemma.
Let $X$ and $Y$ be two sets partitioning $G\setminus\{u, v\}$ as in the definition of having balanced components.  

Suppose that $|G|$ is even.  In this case we have $|X|=|Y|=\frac{1}{2}|G|-1$.  We consider two cases depending on whether $B\cap (N(v)-u)$ is empty or not.

{\bf Case 1:} Suppose that we have $B\cap (N(v)-u) \neq \emptyset$.  Let $v_b$ be a vertex in $B\cap (N(v)-u)$.  Let $P = u v v_b$ to obtain a path such that $G\setminus P$ has balanced components and so  satisfies (\ref{pathdegree2vertices}). 

{\bf Case 2:}  Suppose that we have $B\cap (N(v)-u) = \emptyset$.  Using the integrality of~$|N(v)|$ and~$|B|$, as well as the fact that~$|G|$ is even we obtain that $|N(v)+v|+|B|\geq \left\lceil\frac{1}{2}(|G|+1)\right\rceil+\left\lceil\frac{1}{2}(|G|-1)\right\rceil=|G|+1$.  Since $B\cap (N(v)-u) = \emptyset$ and $v\not\in B$  we have $(N(v)+v)\cap B=\{u\}$. Combined with $|N(v)+v|+|B|\geq |G|+1$ this gives  $d(v)=|B|=\frac{1}{2}|G|$.  Now $P= u$ is a path with $G\setminus P$  having structure (iii) of Lemma~\ref{LemmaLowDegreeDegreeStructures} and so satisfies (\ref{pathdegree2vertices}).

\

Suppose that $|G|$ is odd.  In this case we have $|X|=\frac{1}{2}(|G|-3)$ and $|Y|=\frac{1}{2}(|G|-1)$.  We consider three cases depending on which of $N(v)\cap B\cap X$ and $N(v)\cap B \cap Y$ are empty.

{\bf Case 1:}  Suppose that we have $N(v)\cap B \cap Y \neq \emptyset$.  Let~$b_y$ be a vertex in $N(v)\cap B \cap Y$.  
Let $P = u v b_y$ to obtain a path such that $G\setminus P$ has balanced components and so satisfies (\ref{pathdegree2vertices}). 

{\bf Case 2:}  Suppose that we have $N(v)\cap B \cap X \neq \emptyset$.
Let $b_x$ be a vertex in $N(v)\cap B \cap X$

If $\Delta(Y)<|Y|-1$ holds, then $P = u v b_x$ is a path with $G\setminus P$  having structure (ii) of Lemma~\ref{LemmaLowDegreeDegreeStructures} and so satisfies (\ref{pathdegree2vertices}).

Suppose that $\Delta(Y)=|Y|-1$.  Let $y$ be a vertex in $Y$ with $|Y|-1$ neighbours in $Y$.
There are two subcases depending on whether $N(u)\cap Y$ is empty or not.

{\bf Case 2.1:}  Suppose that $N(u)\cap Y\neq \emptyset$. 
Let $u_y$ be a vertex in $N(u)\cap Y$.   There are two subcases depending on whether $A\cap Y$ is empty or not.

{\bf Case 2.1.1:}
 If $A\cap Y$ contains a vertex $a_y$, then depending on which of $a_y$, $u_y$ and $y$ are equal one of the sequences $a_y y u_y u v b_x$, $a_y u_y u v b_x$ or $a_y u v b_x$ forms a path from $A$ to $B$.  Choose $P$ to be this path in order to obtain a path such that  $G\setminus P$ has balanced components and so satisfies (\ref{pathdegree2vertices}).

{\bf Case 2.1.2:}
If  $A\cap Y= \emptyset$ then we must have $A=X+u$, and in particular $b_x\in A$. Also, $\max(|A|, |B|)\geq \frac{1}{2}|G|$ and $v\not\in A\cup B$ imply that we have $B\cap Y\neq \emptyset$.  Therefore we can repeat the argument for Case 2.1.1, exchanging the roles of $A$ and $B$ to obtain a path satisfying~(\ref{pathdegree2vertices}). 

{\bf Case 2.2:} $N(u)\cap Y = \emptyset$.
Since $G$ does not have a partition into balanced components $X+u+v$ and $Y$, there must be a vertex $v_y$ in $N(v)\cap Y$.

We claim that without loss of generality we can assume that $B\cap Y \neq \emptyset$. Indeed $\max(|A|, |B|)\geq \frac{1}{2}|G|$ and $v\not\in A\cup B$ implies that either $A\cap Y \neq \emptyset$ or $B\cap Y \neq \emptyset$ holds.  If $A\cap Y \neq \emptyset$ holds then we exchange the roles of $A$ and $B$ and of $u$ and $b_x$ for the rest of the proof of this case.
Therefore, we can suppose that $B\cap Y$ contains a vertex, $b_y$.

If $\Delta(X)< |X|-1$, then depending on which of $v_y$, $b_y$ and $y$ are equal, one of the sequences $u v v_y y b_y$, $u v v_y b_y$ or $u v b_y$  forms a path satisfying (\ref{pathdegree2vertices}). 

Therefore we can assume that $\Delta(X)= |X|-1$.  Notice that this implies that $X$ is connected.  Let $x\in X$ be a neighbour of $b_x$.
There are two subcases depending on whether $A\cap Y$ is empty or not.

{\bf Case 2.2.1:}
Suppose that  $A\cap Y=\emptyset$. In this case, we must have $A = X+ u$.  Depending on which of $v_y$, $b_y$ and $y$ are equal, one of the sequences $x b_x v v_y y b_y$, $x b_x v v_y b_y$ or $x b_x v b_y$  forms a path  from $A$ to $B$.  Choose $P$ to be this path in order to obtain a path such that  $G\setminus P$ has balanced components and so satisfies (\ref{pathdegree2vertices}).

{\bf Case 2.2.2:}
Suppose that $A\cap Y$ contains a vertex $a_y$.  

If $a_y$, $v_y$ and $y$ are not all distinct then one of the sequences $P=a_y v_y v b_x$ or $P=a_y v b_x$ forms a path from $A$ to $B$ such that  $G\setminus P$ has balanced components and so satisfies (\ref{pathdegree2vertices}).
  
If $a_y$, $v_y$ and $y$ are all distinct  $\Delta(X-b_x + u)< |X-b_x + u|-1$ then  $P=a_y y v_y v b_x$ is a path with $G\setminus P$  having structure (ii) of Lemma~\ref{LemmaLowDegreeDegreeStructures} and so satisfies (\ref{pathdegree2vertices}).

If $a_y$, $v_y$ and $y$ are all distinct and $\Delta(X-b_x + u)=|X-b_x + u|-1$ then $u$ has a neighbour $u_x\in X\cap(N(x)- b_x +x)$. Depending on whether $u_x=x$ or not, one of the sequences $a_y y v_y v u u_x x b_x$ or $a_y y v_y v u u_x b_x$ forms a path  from $A$ to $B$.  Choose $P$ to be this path to obtain a path such that  $G\setminus P$ has balanced components and so satisfies (\ref{pathdegree2vertices}).

{\bf Case 3:}  Suppose that we have $N(v)\cap B \subseteq \{u\}$.

First we note that, without loss of generality, we can assume that we have either $|B|\geq \frac{1}{2}|G|$ or $u\not \in B$.  Indeed if $u \in B$ and $|B| = \frac{1}{2}(|G|-1)$ both hold, then we must have $|A|\geq \frac{1}{2}|G|$ and so we can exchange the roles of $A$ and $B$ (returning to one of the previous cases of this proof if needed).

Therefore, we can assume that either $|B|\geq \frac{1}{2}|G|$ or $u\not \in B$ holds.  
In either case, we have $|B-u|\geq \frac12(|G|-1)$.
Since $v\not \in B$, and $N(v)\cap B \subseteq \{u\}$ we have that $B-u$ and $N(v)+v$ are disjoint. 
Since $|B-u|\geq \frac12(|G|-1)$ and $|N(v)+v|\geq \frac12(|G|+1)$ we have that  $B-u$ and $N(v)+v$ partition $G$ and satisfy $|B-u|= \frac12(|G|-1)$ and $|N(v)+v|= \frac12(|G|+1)$.
 Since, $G$ does not have a partition into balanced components $B-u$ and $N(v)+v$, there must be an edge between the sets $B-u$ and $N(v)$.
There are two subcases depending on whether this edge intersects $Y$ or not.

{\bf Case 3.1:}
Suppose that there is an edge $v' b_y$ between $v'\in N(v)$ and $b_y\in B\cap Y$.  
If $v'\neq u$, then $P= u v v' b_y$  is a path such that  $G\setminus P$ has balanced components and so satisfies (\ref{pathdegree2vertices}).
If $v'=u$, then $P=u b_y$ is a path such that $G\setminus P$  has structure (iii) of Lemma~\ref{LemmaLowDegreeDegreeStructures} and so satisfies (\ref{pathdegree2vertices}).

{\bf Case 3.2:}
If there are no edges between $N(v)$ and $B\cap Y$ then there must be an edge $v' b_x$ between $v'\in N(v)$ and $b_x\in B\cap X$. 

In particular this means that $B\cap X$ is not empty, which together with $|B-u|=\frac 12(|G|-1)$ and $|Y|=\frac 12(|G|-1)$ gives $Y\setminus (B-u)\neq \emptyset$. Since $B-u$ and $N(v)+v$ partition $G$, we have $N(v)\cap Y \neq \emptyset$.  

Since $|N(v)+v|=\frac{1}{2}(|G|+1)$ and $u,v \in N(v)+v$ we $|N(v)\cap Y|\leq |(N(v)+v)\setminus \{u,v\}|=\frac{1}{2}(|G|-3)<|Y|$.
Since $B-u$ and $N(v)+v$ partition $G$, we have $B\cap Y \neq \emptyset$.  

Since $N(v)\cap Y \neq \emptyset$ and $B\cap Y \neq\emptyset$ both hold and there are no edges between $N(v)\cap Y$ and $B\cap Y$, we have that $\Delta(Y)< |Y|-1$.  
If $v'=u$, then $P=u v v' b_x$ is a path such that $G\setminus P$  has structure (ii) of Lemma~\ref{LemmaLowDegreeDegreeStructures} and so satisfies (\ref{pathdegree2vertices}).
If  $v'=u$ then $P=u b_x$ is a path such that $G\setminus P$  has structure (iv) of Lemma~\ref{LemmaLowDegreeDegreeStructures} and so satisfies (\ref{pathdegree2vertices}).
\end{proof}

We are now ready to prove Lemma~\ref{mainlemma}. The proof is very similar to the proof of Lemma~\ref{PathMinDegreeSimple}.

\begin{proof}[Proof of Lemma~\ref{mainlemma}]
The proof is by induction on the number of vertices of $G$.  If $|G|=2$ or $3$ the lemma holds trivially, taking $P$ to be either an edge between $A$ and $B$ or a single vertex in $A \cap B$.  At least one of these exist since $G$ does not have balanced components.

For $n\geq 4$, suppose that the lemma holds for all graphs with at most $n-1$ vertices.  
Let $G$ be a graph with $n$ vertices. We split into two cases depending on the maximum degree of $G$.

{\bf Case 1:}
Suppose that there is a vertex $x\in G$ such that $d(x)\geq \frac{1}{2}(|G|-1)$.  We consider three subcases depending on whether $x$ is in $A$, $x$ is in $B$, or $x$ is in neither $A$ nor $B$.

{\bf Case 1.1:}
Suppose that $x\in A$ holds.  
Let $G'=G-x$, $A'=B- x$, and $B'=N(x)$.  Note that we  have $|A'|\geq |B|-1\geq \frac{1}{2}|G|-1=\frac{1}{2}(|G'|-1)$ and $|B'|=|N(x)|\geq \frac{1}{2}(|G|-1)= \frac{1}{2}|G'|$.

If $G'$ has balanced components, then the lemma follows from Lemma~\ref{1vertex} (applied with $v=x$, $A=A$, and $B=B$.)  Otherwise, if $G'$ does not have balanced components  we can apply induction in order to find a path $P'$ in $G'$ starting in $B'$ and ending in $A'$ such that $\Delta(G'\setminus P') \leq \frac{1}{2}(|G'\setminus P'|-1)$.  The path $P=x+P'$ is from $A$ to $B$ and satisfies~(\ref{pathdegree}).

{\bf Case 1.2:}
Suppose that $x\not\in A$ and $x \in B$ holds.  
Let $G'=G-x$, $A'=A$, and $B'=N(x)$.  Note that we have $|A'|\geq\frac{1}{2}(|G'|-1)$ and $|B'|=|N(x)|\geq \frac{1}{2}(|G|-1)=\frac{1}{2}|G'|$. 

If $G'$ has balanced components, then the lemma follows from Lemma~\ref{1vertex} (applied with $v=x$, $\tilde A=B$, and $\tilde B=A$.)  Otherwise, if $G'$ does not have balanced components  we can apply induction in order to find a path $P'$ in $G'$ starting in $A'$ and ending in $B'$ such that $\Delta(G'\setminus P') \leq \frac{1}{2}(|G'\setminus P'|-1)$.  The path $P=P'+x$ is from $A$ to $B$ and satisfies (\ref{pathdegree}).

{\bf Case 1.3:}
Suppose that $x\not\in A$ and $x \not\in B$ holds.  
Note that since $N(x)$ and $B$ are both contained in $G- x$ and $|N(x)|+|B|\geq |G|-\frac{1}{2}$, we have that $N(x)\cap B$ is nonempty.  

Suppose that $N(x)\cap A$ contains a vertex, $y$.
Let $G'=G- x-y$, $A'=N(x)-y$, and $B'=B- y$.  Note that we have $|A'|=|N(x)|-1\geq\frac{1}{2}(|G|-3)=\frac{1}{2}(|G'|-1)$ and $|B'|\geq |B|-1\geq \frac{1}{2}|G|-1=\frac{1}{2}|G'|$.  
If $G'$ has balanced components, then the lemma follows from Lemma~\ref{2vertices} (applied with $v=x$, $u=y$, $A=A$, and $B=B$.)  Otherwise, if $G'$ does not have balanced components  we can apply induction in order to find a path $P'$ in $G'$ starting in $A'$ and ending in $B'$ such that $\Delta(G'\setminus P') \leq \frac{1}{2}(|G'\setminus P'|-1)$.  The path $P=y+x+P'$ is from $A$ to $B$  and satisfies~(\ref{pathdegree}).

Suppose that $N(x)\cap A$ is empty. Let $z$ be a vertex in $N(x)\cap B$.   Let $G'=G-x-z$, $A'=N(x)-z$, and $B'=A$.  Note that we have  $|A'|=|N(x)|-1\geq\frac{1}{2}(|G|-3)=\frac{1}{2}(|G'|-1)$ and $|B'|=|A|\geq \frac{1}{2}(|G|-1)>\frac{1}{2}|G'|$.  
If $G'$ has balanced components, then the lemma follows from Lemma~\ref{2vertices} (applied with $v=x$, $u=z$, $\tilde A=B$, and $\tilde B=A$.)  Otherwise, if $G'$ does not have balanced components  we can apply induction in order to find a path $P'$ in $G'$ starting in $A'$ and ending in $B'$ such that $\Delta(G'\setminus P') \leq \frac{1}{2}(|G'\setminus P'|-1)$.  The path $P=z+x+P'$ is from $B$ to $A$  and satisfies~(\ref{pathdegree}).

{\bf Case 2:}
Suppose that $\Delta(G)<\frac{1}{2}(|G|-1)$.    If there is a vertex in $A\cap B$ let $P$ be this vertex. We have that $\Delta(G\setminus P) < \frac{1}{2}(|G|-1)= \frac{1}{2}|G\setminus P|$.  Using the fact that $|G\setminus P|$ and $\Delta(G\setminus P)$ are both integers, we obtain that (\ref{pathdegree}) holds.  

Therefore, we can suppose that $A$ and $B$ are disjoint.  Notice that in this case we have $|A|+|B|\geq |G|-\frac{1}{2}$ which implies that $A\cup B= V(G)$.

Suppose there is a vertex $x$ of degree $\frac{1}{2}|G|-1$ in $G$.  Let $G'=G- x$ and $A'=N(x)$.  Let $B'=A$ if $x \in B$ and $B'=B$ if $x \in A$.  In either case we have $|A'|=|N(x)|= \frac{1}{2}|G|-1=\frac{1}{2}(|G'|-1) $ and $|B'|\geq |A| \geq \frac{1}{2}(|G|-1)=\frac{1}{2}|G'|$.  
If $G'$ has balanced components, then the lemma follows from Lemma~\ref{1vertex}.  Otherwise, if $G'$ does not have balanced components  we can apply induction in order to find a path $P'$ in $G'$ starting in $A'$ and ending in $B'$ such that $\Delta(G'\setminus P') \leq \frac{1}{2}(|G'\setminus P'|-1)$.  The path $P=P'+x$ is between $A$ and $B$ and satisfies~(\ref{pathdegree}).  

Suppose that we have $\Delta(G)< \frac{1}{2}|G|-1$.  Since $G$ does not have balanced components, there must be an edge between $A$ and $B$. Let $P$ be this edge.  We have that $\Delta(G\setminus P) \leq \Delta(G)< \frac{1}{2}|G|-1= \frac{1}{2}|G\setminus P|$.  Using the fact that $|G\setminus P|$ and $\Delta(G\setminus P)$ are both integers, we obtain that (\ref{pathdegree}) holds. 
 \end{proof}

We are now ready to deduce Theorem~\ref{CycleMinDegree}.  

\begin{proof}[Proof of Theorem~\ref{CycleMinDegree}]
If $\Delta(G)\leq \frac{1}{2}(|G|-1)$ then letting $C=\emptyset$ gives a cycle satisfying (\ref{EqCycleMinDegree}).

Otherwise, there must be a vertex $v \in G$ of degree at least $\frac{1}{2}|G|$.  If $G-v$ has balanced components then we can let $C=v$ to obtain a cycle  satisfying (\ref{EqCycleMinDegree}).  If $G-v$ does not have balanced components then we can apply Lemma~\ref{mainlemma} to the graph $G-v$ with $A=B=N(v)$ to obtain a path $P$ from $N(v)$ to $N(v)$ satisfying $\Delta(G\setminus P-v)\leq \frac{1}{2}(|G\setminus P-v|-1)$.  We can close the path with the vertex $v$ to obtain a cycle satisfying (\ref{EqCycleMinDegree}).
\end{proof}

\section{Partitioning highly connected graphs}\label{SectionConnected}
\newcommand{\CNhood}{\overline{N}}
In this section we prove Theorem~\ref{CycleMinDegreeConn}.

Recall that the degree of a vertex $v\in G$ is the number of edges containing it, denoted $d(v)$. If we have a set $S\subseteq V(G)$, then $d_S(v)$ denotes the degree of $v$ in the induced subgraph $G[S+v]$. For a subgraph $H$ of $G$, it will be convenient to use $d_H(v)$ to mean $d_{V(H)}(v)$. 


\begin{proof}[Proof of Theorem~\ref{CycleMinDegreeConn}]
For a set $S\subseteq V(G)$ and a vertex $v\in S$, let $C_S(v)$ be the connected component of $G[S]$ containing~$v$.  Let $f(S)=|\{v\in S: \Delta(C_S(v))=\Delta(S)\}|$ i.e. $f(S)$ is the total number of vertices contained in components of $S$ with a vertex of maximal degree.

Let $C$ be a cycle in $G$ satisfying the following.
\begin{enumerate} [\normalfont(i)]
 \item $\Delta(G\setminus C)$ is as small as possible.
\item $f(G\setminus C)$ is as small as possible, whilst keeping (i) true.
\item $|C|$ is as small as possible, whilst keeping (i) and (ii) true.
\end{enumerate}
We will show that the  cycle $C$  satisfies the theorem.
Let the sequence of vertices of $C$ around the cycle be $c_1,c_2, \dots, c_{|C|}$. A \emph{clockwise} sequence of vertices in $C$ is one of the form $c_{i \pmod{|C|}},$ $c_{i+1 \pmod{|C|}},$ $c_{i+2 \pmod{|C|}}, \dots, c_{i+t \pmod{|C|}}$.

For $v\in C$, let  
\begin{align*}
\CNhood(v)&=C_{G\setminus C+v}(v)-v\\
&=\{u\in G\setminus C:  \text{There is a path between } u \text{ and } v \text{ contained in } G\setminus C+v\}.
\end{align*}
We say that $v\in C$ \emph{has large neighbourhood} if $|\CNhood(v)|\geq \Delta(G\setminus C)-3$.  Otherwise we say that $v$ has \emph{small neighbourhood}.
We'll need the following technical claim.

\begin{claim}\label{ClaimDegreeInvariance}
 Let $C'$ be a cycle such that every $v \in C\setminus C'$ has small neighbourhood and every $v\not\in C'$ satisfies $|N(v)\cap (C\setminus C')|\leq 3$.  The following hold
\begin{enumerate}[(1)]
\item Let $u$ be a vertex outside $C'$.
 Then  either we have $d_{G\setminus C'}(u)<\Delta(G\setminus C)$, or we have $u\not\in C$ and
 \begin{equation}\label{DegreeInvariance}
  d_{G\setminus C'}(u)= d_{G\setminus C}(u)= d_{G\setminus (C\cup C')}(u).
 \end{equation}
\item $\Delta(G\setminus C')= \Delta(G\setminus C)=\Delta(G\setminus (C\cup C'))$.
\item $f(G\setminus C')= f(G\setminus C)=f(G\setminus (C\cup C'))$.
\end{enumerate} 
 \end{claim}
\begin{proof}
First we prove (1).
Suppose for the sake of contraadiction that we have a vertex $u$ outside of $C'$ with $d_{G\setminus C'}(u)\geq\Delta(G\setminus C)$. 

Suppose that $u\in C$.  Then the following holds.
$$|\CNhood(u)|\geq |N(u)\setminus C|\geq |N(u)\setminus (C\cup C')|=d_{G\setminus C'}(u)- |N(u)\cap (C\setminus C')|\geq \Delta(G\setminus C)-3. $$
This contradicts $u\in C\setminus C'$ having small neighbourhood.

Therefore, we can suppose  that $u\not \in C$.  
If there is some $x \in N(u)\cap (C\setminus C')$, then note that we have $N(u)\setminus C \subseteq  \CNhood(x)$, which implies that
$$|\CNhood(x)|\geq |N(u)\setminus C|\geq |N(u)\setminus (C\cup C')|=d_{G\setminus C'}(u)- |N(u)\cap (C\setminus C')|\geq \Delta(G\setminus C)-3.$$
This contradicts $x\in C\setminus C'$ having small neighbourhood.

If $N(u)\cap (C\setminus C')=\emptyset$ then  we have $N_{G\setminus C'}(u)\subseteq (N(u)\setminus C)\cup (N(u)\cap (C\setminus C'))= N_{G\setminus C}(u)$. This implies $N_{G\setminus (C'\cup C')}(u)=N_{G\setminus C}(u)\cap N_{G\setminus C'}(u)=N_{G\setminus C}(u)$ which gives $d_{G\setminus C}(u)= d_{G\setminus (C\cup C')}(u)$. 
Combining $N_{G\setminus C'}(u)\subseteq N_{G\setminus C}$ with $d_{G\setminus C'}(u)\geq \Delta(G\setminus C)\geq d_{G\setminus C}(u)$ gives $d_{G\setminus C'}(u)=d_{G\setminus C}(u)$ which implies (\ref{DegreeInvariance}).

For (2), let $u$ be any vertex in $G\setminus C'$ with $d_{G\setminus C'}(u)=\Delta(G\setminus C')$. 
By the minimality of $C$ in (i) we have $\Delta(G\setminus C')\geq \Delta(G\setminus C)$ which implies $d_{G\setminus C'}(u)\geq \Delta(G\setminus C)$. Using part (1), we have $\Delta(G\setminus C) \leq d_{G\setminus C'}(u)= d_{G\setminus C}(u)= d_{G\setminus (C\cup C')}(u)$ and also $u\not\in C$. In particular we have $\Delta({G\setminus (C\cup C')})\geq d_{G\setminus (C\cup C')}(u) \geq \Delta({G\setminus C})$ which combined with $\Delta({G\setminus (C\cup C')})\leq \Delta({G\setminus C})$ gives the equality $\Delta({G\setminus (C\cup C')})= \Delta({G\setminus C})$.
Since $u\not\in C$ we get $d_{G\setminus C}(u)\leq \Delta(G\setminus C)$ which implies $d_{G\setminus C'}(u)= d_{G\setminus C}(u)=\Delta(G\setminus C)$. Since $u$ was chosen to have $d_{G\setminus C'}(u)=\Delta(G\setminus C')$ we obtain the second equality $\Delta({G\setminus  C'})= \Delta({G\setminus C})$.

For (3), we use the fact that  (2) tells us that  $G\setminus (C\cup C')$  and $G\setminus C'$ have maximum degree $\Delta(G\setminus C)$.  By part (1), any vertex of degree $\Delta(G\setminus C)$ in $G\setminus C'$ also has degree $\Delta(G\setminus C)$ in both $G\setminus (C\cup C')$  and $G\setminus C$. This implies that $f(G\setminus C')\leq f(G\setminus (C\cup C'))$ and $f(G\setminus C')\leq f(G\setminus C)$. Now part (3) follows from the fact that we have have $f(G\setminus C')\geq f(G\setminus C)$ (from the minimality of $C$ in (ii)) and $f(G\setminus C')\geq f(G\setminus (C\cup C'))$.
\end{proof}

Let $A$ be an arbitrary component of $G\setminus C$ satisfying $\Delta(A)=\Delta(G\setminus C)$.
We say that $v \in C$ is \emph{connected to $A$} if there is an edge between $v$ and $A$. 
Notice that if $v$ is connected to $A$, then it has large neighbourhood. 
The next two claims prevent certain sequences of vertices from existing on the cycle $C$.
\begin{claim}\label{1sequence}
 There does not exists a sequence of vertices $x_1, \dots, x_m$ along $C$ such that all the vertices $x_2, \dots, x_{m-1}$ have small neighbourhood and at least one of the following holds.
\begin{enumerate}[(I)]
 \item $m\geq 3$ and $x_1 x_m$ is an edge.
 \item $m\geq 4$ and  $N(x_1)\cap N(x_m)\setminus C\neq \emptyset$.
 \item $x_1$ and $x_m$ are both connected to $A$
\end{enumerate}
\end{claim}

\begin{proof}
Suppose, for the sake of contradiction, that $x_1,\dots ,x_m$ is such a sequence.
 We can assume that $x_1,\dots, x_m$ is a minimal such sequence i.e. no   sequence $x_i, x_{i+1}\dots, x_j$ satisfies any of (I) -- (III) for $i<j<m$ or $1<i<j$.

In each of the three cases we define a cycle $C'$ to which we apply Claim \ref{ClaimDegreeInvariance}.

If (I) holds, let $C' = C\setminus \{x_2, \dots, x_{m-1}\}$ with the edge $x_1x_m$ added.

If (II) holds, let $u$ be a vertex in $N(x_1)\cap N(x_m)\setminus C$ and $C'=C\setminus \{x_2, \dots, x_{m-1}\} + u$ with the edges $x_1u$ and $u x_m$ added.

If (III) holds, then note that there must be a path $P$ contained in $A$ such that the start of $P$ is joined to $x_1$ and the end of $P$ is joined to $x_m$.  Let $C'=(C\cup P)\setminus \{x_2, \dots, x_{m-1}\}$.

We claim that in either case, the conditions of Claim~\ref{ClaimDegreeInvariance} are satisfied.  In each of (I) -- (III),  we have $C\setminus C'=\{x_2, \dots, x_{m-1}\}$, and so all the vertices in $C\setminus C'$ have small neighbourhood by the assumption of the lemma.
Suppose that a vertex $v\not\in C'$ has $4$ or more neighbours in $\{x_2, \dots, x_{m-1}\}$. 
If $v\in C$, then we would have a shorter sequence satisfying~(I).
If $v\not \in C$, then we would have a shorter sequence satisfying (II).

Therefore the conclusion of Claim~\ref{ClaimDegreeInvariance} holds for $C'$, and in particular $\Delta(G\setminus C')= \Delta(G\setminus C)=\Delta(G\setminus(C\cup C'))$ and  $f(G\setminus C')= f(G\setminus C)=f(G\setminus (C\cup C'))$.   If (I) or (II) holds then $|C'|<|C|$ contradicting minimality of $C$ in (iii).
If (III) holds, then, since $\Delta(A)=\Delta(G\setminus C)$, we have $f(G\setminus (C\cup C'))= f(G\setminus C)-|P|$ which contradicts part (3) of Claim~\ref{ClaimDegreeInvariance}.
\end{proof}

The following claim is similar to Claim~\ref{1sequence}, except there are now two sequences $x_1,  \dots, x_i$ and $y_1, \dots, y_j$ playing the role that  $x_1,  \dots, x_m$ played in Claim~\ref{1sequence}.
\begin{claim}\label{2sequences}
There do not exist two disjoint clockwise sequences $x_1, x_2, \dots, x_i$ and $y_1, y_2, \dots, y_j$ along $C$ such that $x_1$ and $y_1$ are both connected to $A$, the vertices $x_2, \dots, x_{i-1}$, $y_2\dots, y_{j-1}$ all have small neighbourhood, and one of the following holds.
\begin{enumerate}[(a)]
 \item $x_iy_j$ is a red edge.
 \item $\CNhood(x_i)\cap\CNhood(y_j) \neq \emptyset$.
\end{enumerate}
\end{claim}
\begin{proof}
Suppose for the sake of contradiction that $x_1, x_2, \dots, x_i$ and $y_1, y_2, \dots, y_j$  are such sequences. 
 We can  assume that  $x_1, x_2, \dots, x_i$ and $y_1, y_2, \dots, y_j$  are minimal such sequences sequences,  i.e. any pair of  clockwise subsequences $x_a, x_{a+1}, \dots, x_b$ and $y_c, y_{c+1}, \dots, y_d$ satisfying~(a) or~(b) have $(x_a,x_b,y_c,y_d)=(x_1,x_i,y_1,y_j)$.
Using Claim~\ref{1sequence} we can also assume that neither of these sequences satisfy (I), (II), or (III) of Claim~\ref{1sequence}.

Note that there is a path $P\subseteq A$ such that the start of $P$ is connected to $x_1$ and the end of $P$ is connected to $y_1$.  Let $p_x$ be the start of $P$ and $p_y$ the end of $P$.
We construct a cycle $C'$ satisfying  the conditions of Claim~\ref{ClaimDegreeInvariance}.

If (a) holds, let $C'=(C\cup P)\setminus\{x_2, \dots, x_{i-1}, y_2, \dots, y_{j-1}\}$.  By adding the edges $x_iy_j, p_x x_1$, and $p_y y_1$, notice that this is indeed a cycle.

If (b) holds, note that there is a path $Q$ contained in $G\setminus C$ between $x_i$ and $y_j$.  The paths $P$ and $Q$ are disjoint since from part (III) of Claim~\ref{1sequence}, we have that neither $x_i$ nor $y_j$ is connected to $A$.  Therefore, by joining $P$ and $Q$ to $C\setminus \{x_2, \dots, x_{i-1}, y_2, \dots, y_{j-1}\}$, we can find a cycle $C'$ with vertex set $(C\cup P\cup Q)\setminus\{x_2, \dots, x_{i-1}, y_2, \dots, y_{j-1}\}$.

We claim that in either case, the conditions of Claim~\ref{ClaimDegreeInvariance} are satisfied.  
  In each of (a) and (b),  we have $C\setminus C'=\{x_2, \dots, x_{i-1}, y_2, \dots, y_{j-1}\}$, and so all the vertices in $C\setminus C'$ have small neighbourhood by the assumption of the lemma.
Let $v$ be a vertex not in $C'$. We need to show that $|N(v)\cap (C\setminus C')|\leq 3$.

Suppose that $v\not\in C$. Then $v$ cannot have neighbours in both $\{x_2, \dots, x_{i-1}\}$ and $\{y_2,\dots,y_{j-1}\}$ since otherwise we would have a shorter pair of sequences satisfying (b). Also, $v$ can have at most $3$ neighbours in each of $\{x_2, \dots, x_{i-1}\}$ and $\{y_2, \dots, x_{j-1}\}$ since otherwise we would have a sequence satisfying the condition (II) of Claim~\ref{1sequence}. Therefore $|N(v)\cap (C\setminus C')|\leq 3$.

Suppose that $v\in \{x_2, \dots, x_{i-1}\}$.  Then $v$ can have at most $2$ neighbours in $\{x_2, \dots, x_{i-1}\}$ since no sequence of vertices in $C$ satisfying condition (I) of Claim~\ref{1sequence}.  Also $v$ cannot have any neighbours in $\{y_2, \dots, y_{j-1}\}$ since otherwise $x_1, \dots, v$ and $y_1, \dots, y_j$ are a shorter pair of sequences satisfying (a). Therefore $|N(v)\cap (C\setminus C')|\leq 3$.

Similarly if $v\in \{y_2, \dots, y_{j-1}\}$ then it has at most $2$ neighbours in $\{y_2, \dots, y_{j-1}\}$ and no neighbours in $\{x_2, \dots, x_{i-1}\}$, implying $|N(v)\cap (C\setminus C')|\leq 3$. 

Therefore the conclusion of Claim~\ref{ClaimDegreeInvariance} holds for $C'$, and in particular we have $\Delta(G\setminus C')= \Delta(G\setminus C)=\Delta(G\setminus(C\cup C'))$ and  $f(G\setminus C')= f(G\setminus C)=f(G\setminus (C\cup C'))$.  
Since $\Delta(A)=\Delta(G\setminus C)$, we have $f(G\setminus (C\cup C'))\leq f(G\setminus C)-|P|$ which contradicts part (3) of Claim~\ref{ClaimDegreeInvariance}.
\end{proof}

We now prove the theorem.
Notice that as a consequence of part (III) of Claim \ref{1sequence}, there is at least one vertex  $c\in C$ which is not connected to $A$ 
(since otherwise, for $|C|\geq 2$, we would have two adjacent vertices connected to $A$, which is excluded by part (III) of Claim \ref{1sequence}. If $|C|=1$, then joining the vertex in $C$ to any $a \in A$ gives a new cycle with either smaller $\Delta(G\setminus C)$ or $f(G\setminus C)$).

By $k$-connectedness of $G$ there are $k$ vertices, $a_1, \dots, a_k$, on $C$ which are connected to~$A$ (since the set of vertices on $C$ connected to $A$ form a cutset separating $c$ from $A$).
For each $i\in \{1, \dots, k\}$, let $b_i$ be the first clockwise vertex after $a_i$ such that $b_i$ has large neighbourhood. Since all the vertices $a_i$ have large neighbourhood, the vertices $b_1, \dots, b_k$ are well-defined and distinct.

If for any $i$, $b_i$ is connected to $A$, then the clockwise sequence between $a_i$ and $b_i$ satisfies condition (III) of Claim~\ref{1sequence}, leading to a contradition.
Therefore, we can assume that $b_i$ is not connected to $A$. In particular, this implies that $b_i$ lies  inside the clockwise interval between $a_{i}$ and $a_{i+1 \pmod k}$, and also $b_i\neq a_i, a_{i+1 \pmod k}$.

Suppose that two of the sets $A, \CNhood(b_1), \dots, \CNhood(b_k)$ intersect.
 We have $A\cap\CNhood(b_i)=\emptyset$  for all $i$  since $b_i$ is not connected to $A$.
If for any $i\neq j$, we have $\CNhood(b_i)\cap\CNhood(b_j)\neq \emptyset$, then the clockwise sequences between $a_i$ and $b_i$, and between $a_j$ and $b_j$ satisfy condition~(b) of Claim~\ref{2sequences}, which is a contradiction.

Therefore, we have that the sets $A, \CNhood(b_1), \dots, \CNhood(b_k)$ are all disjoint.  This, together with $|A|\geq \Delta(G\setminus C)+1$ and $|\CNhood(b_i)|\geq \Delta(G\setminus C)-3$ imply that we have 
$$|G\setminus C|\geq |A|+ |\CNhood(b_1)|+ \dots+|\CNhood(b_k)| \geq (k+1)\Delta(G\setminus C)- 3 k+1.$$
This implies that we have (\ref{EqCycleMinDegreeConn}), proving the theorem.
\end{proof}

\section{An application: Partitioning a  3-edge-coloured $K_n$ into 3 monochromatic cycles}\label{Section3Cycles}
\newcommand{\Ar}{A_{b,g}}
\newcommand{\Ab}{A_{r,g}}
\newcommand{\Ag}{A_{r,b}}
\newcommand{\Apr}{A'_{b,g}}
\newcommand{\Apb}{A'_{r,g}}
\newcommand{\Apg}{A'_{r,b}}
\newcommand{\Cr}{W_{b,g}}
\newcommand{\Cb}{W_{r,g}}
\newcommand{\Cg}{W_{r,b}}
\newcommand{\Cpr}{W'_{b,g}}
\newcommand{\Cpb}{W'_{r,g}}
\newcommand{\Cpg}{W'_{r,b}}
\newcommand{\Dr}{D_{b,g}}
\newcommand{\Db}{D_{r,g}}
\newcommand{\Dg}{D_{r,b}}
\newcommand{\Dpr}{D'_{b,g}}
\newcommand{\Dpb}{D'_{r,g}}
\newcommand{\Dpg}{D'_{r,b}}
\newcommand{\Er}{E_{b,g}}
\newcommand{\Eb}{E_{r,g}}
\newcommand{\Eg}{E_{r,b}}
\newcommand{\Fr}{F_{b,g}}
\newcommand{\Fb}{F_{r,g}}
\newcommand{\Fg}{F_{r,b}}
\newcommand{\cCases}{41002}
\newcommand{\mCases}{400}
\newcommand{\kCases}{4}
\newcommand{\cDisconnecting}{c_{d}}
\newcommand{\DSize}{400}
\newcommand{\halfD}{30}
\newcommand{\DminusPaths}{130}
\newcommand{\Bipart}[3]{#1[#2, #3]}
The goal of this section is to prove Theorem~\ref{ThreeCycles}. 
When the colouring of $K_n$ is highly connected in some colour (say red), then the approach to proving Theorem~\ref{ThreeCycles} is very simple: We treat blue and green as one colour, and use Theorem~\ref{CycleMinDegreeConn} to partition $K_n$ into a red cycle $C$ and a blue-green graph $H$ with high minimum degree. 
Then we use the following theorem to partition $H$ into two disjoint monochromatic cycles.
\begin{theorem}[Letzter, \cite{LetzterMinDegree}]\label{Letzter}
There is an $n_0$ such that every $2$-edge-coloured graph $G$ of order at least $n_0$  and $\delta(G)\geq \frac{3}{4}|G|$ can be covered by $2$ vertex-disjoint monochromatic cycles with different colours.
\end{theorem}  
For large $|G|$, the above theorem solved a conjecture of Balogh, Bar\'at, Gerbner, Gy\'arf\'as, and S\'ark\"ozy \cite{Balogh}.
Previously  Balogh, Bar\'at, Gerbner, Gy\'arf\'as, and S\'ark\"ozy \cite{Balogh} proved an approximate version of Theorem~\ref{Letzter}---that every $G$ with $\delta(G)\geq (3/4+o(1))|G|$ contains two disjoint monochromatic cycles covering $(1-o(1))|G|$ vertices. Subsequently DeBiaso and Nelsen~\cite{DeBiasioNelsen} proved a stronger approximate version---that every $G$ with $\delta(G)\geq (3/4+o(1))|G|$ contains two disjoint monochromatic cycles covering all the vertices. 
For our purposes we could use DeBiaso and Nelsen's result instead of Theorem~\ref{Letzter} in the proof of Theorem~\ref{ThreeCycles}.

The case when $K_n$ is not highly connected in any colour takes some extra work. To deal with this case, we first  prove a lemma classifying the structure of such colourings (Lemma~\ref{CasesLemma}.)

We introduce some notation which will help us deal with coloured graphs throughout this section.
For two sets of vertices $S$ and $T$ in a graph $G$, let $\Bipart{G}{S}{T}$ be the subgraph of~$G$ with vertex set $S\cup T$ with $st$ an edge of $\Bipart{G}{S}{T}$ whenever $s\in S$ and~$t \in T$.
If a graph $G$ is coloured with some number of colours we define the \emph{{red} colour class} of $G$ to be the subgraph of $G$ with vertex set $V(G)$ and edge set consisting of all the {red} edges of $G$.  
We say that $G$ is $k$-\emph{connected in} {red}, if the {red} colour class is a $k$-verex-connected graph. Similar definitions are made for the colours blue and green as well. 
We will need the following special 3-colourings of the complete graph.
\begin{definition}\label{4partitedefinition}
 Suppose that the edges of $K_n$ are coloured with three colours.  We say that the colouring is \textbf{4-partite} if there exists a partition of the vertex set into four nonempty sets $A_1$, $A_2$, $A_3$, and $A_4$ such that the following hold.
\begin{itemize}
 \item The edges between $A_1$ and $A_4$, and the edges between $A_2$ and $A_3$  are {red}.
 \item The edges between $A_2$ and $A_4$, and the edges  between  $A_1$ and $A_3$ are {blue}. 
 \item The edges between $A_3$ and $A_4$, and the edges  between $A_1$ and $A_2$ are {green}. 
\end{itemize}
 The edges within the sets $A_1$, $A_2$, $A_3$, and $A_4$ can be coloured arbitrarily. 
\end{definition}
See Figure~\ref{figure4partite} for an illustration of a 4-partite colouring of $K_n$.  The sets $A_1$, $A_2$, $A_3$, and $A_4$  in Definition~\ref{4partitedefinition} will be called the ``classes" of the $4$-partition. 
\begin{figure}
  \centering
    \includegraphics[width=0.4\textwidth]{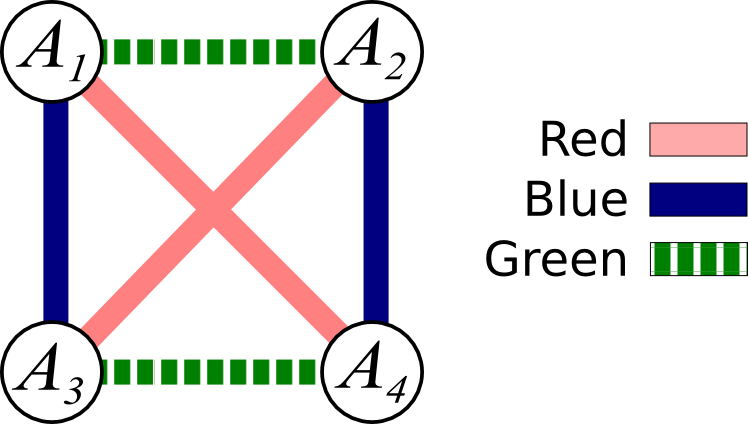}
  \caption{A 4-partite colouring of $K_n$.} \label{figure4partite}
\end{figure}
The following lemma gives a useful alternative characterization of $4$-partite colourings of $K_n$.
\begin{lemma} \label{4partite}
Suppose that the edges of $K_n$ are coloured with three colours.  The colouring is 4-partite if and only it is disconnected in each colour and there is a {red} connected component $C_1$ and a {blue} connected component $C_2$ such that all of the sets $C_1\cap C_2$, $(V(K_n)\setminus C_1)\cap C_2$, $C_1\cap (V(K_n)\setminus C_2)$, and $(V(K_n)\setminus C_1)\cap (V(K_n)\setminus C_2)$ are nonempty.
\end{lemma}
\begin{proof}
Suppose that we have a {red} component $C_1$ and a {blue} component $C_2$ as in the statement of the lemma.  Let  $A_1=C_1\cap (V(K_n)\setminus C_2)$, $A_2=(V(K_n)\setminus C_1)\cap C_2$, $A_3=(V(K_n)\setminus C_1)\cap (V(K_n)\setminus C_2)$, and $A_4=C_1\cap C_2$.  

Since $C_1$ and $C_2$ are red and blue components respectively, all the edges between $A_1$ and $A_2$ and between $A_3$ and $A_4$ are green.  Since $K_n$ is not connected in green, there cannot be any green edges between $A_1$ and $A_3$.  Therefore, since $A_1\subseteq C_1$ and $A_3\cap C_1=\emptyset$, all the edges between $A_1$ and $A_3$ are blue.  Similarly, the edges between $A_1$ and $A_4$ are all red.  Since $K_n$ is not connected in red or green, the edges between $A_2$ and $A_4$ are all blue.  Since $K_n$ is not connected in blue or green, the edges between $A_2$ and $A_3$ are all red.
This ensures that the sets $A_1$, $A_2$, $A_3$, and $A_4$ form the classes of a 4-partite colouring of $K_n$.

For the converse, suppose that $A_1$, $A_2$, $A_3$, and $A_4$  form the classes of a 4-partite colouring.  Choose $C_1=A_1\cup A_4$ and $C_2=A_2 \cup A_4$ to obtain components as in the statement of the lemma.  
\end{proof}

We'll need the following lemma.
\begin{lemma}\label{LemmaManyComponents}
For every $k$ and $m$, every graph $G$ has a set $S\subseteq V(G)$ with $|S|\leq (k-1)m$ such $G\setminus S$ is partitioned  into connected componets $C_1, \dots, C_t$ satisfying one of the following.
\begin{enumerate}[\normalfont(a)]
\item $t\geq m$.
\item For each $t\in \{1,\dots,  t\}$, $C_i$ is either $k$-connected or complete.
 \end{enumerate}
\end{lemma}
\begin{proof}
The proof is by induction on $m$. The initial case $m=1$ is true since we can take $S=\emptyset$ and $C_1=V(G)$ which satisfies (a).

Suppose that the lemma is true for some $m\geq 1$. Let $G$ be a graph.  We will show that $G$ has a paritition onto $S, C_1, \dots, C_t$ satisfying (a) or (b) for $m'=m+1$.

By induction we have a set $S$ with $|S|\geq (k-1)m$ such that $G\setminus S$ has a partition into connected componets $C_1, \dots, C_t$, with satisfying either (a) or (b). If  $C_1, \dots, C_t$ satisfy (b) or if $t \geq m+1$, then we are done.
Therefore we have that $t=m$, and there is some $i$ for which $C_i$ is neither $k$-connected nor complete. Since $C_i$ is neither $k$-connected nor complete $C_i$ has a cutset $T$ with $|T|\leq k-1$ separating $C_i$ into at least $2$ components. Now $S'=S\cup T$ is a set with $|S'|\leq (k-1)(m+1)$ such that $G\setminus S$ has a $\geq m+1$ components and so satisfies (a).
\end{proof}

A corollary of the above  lemma is that every sufficiently large graph $G$ contains either a large $k$-connected subgraph or the complement of $G$ contains a large complete bipartite graph. Let $Conn(G)$ be the order of the largest connected component of $G$.
\begin{lemma}\label{Disconnecting}
For every $k$ and $m$, there exists a  constant $\cDisconnecting(k,m)\leq km$  such that the following holds.
Every graph $G$ contains a set $T$ of vertices such that $|T|\leq \cDisconnecting(k,m)$, and one of the following holds.
\begin{enumerate}[\normalfont(i)]
\item $G\setminus T$ is either $k$-connected or complete.
\item $Conn(G\setminus T)\leq |G\setminus T|-m$.
 \end{enumerate}
\end{lemma}
\begin{proof}
Apply Lemma~\ref{LemmaManyComponents} to $G$ with in order to get a set $S$ with $|S|\leq (k-1)m$ such that $G\setminus S$ has a partition into components $C_1, \dots, C_t$ satisfying (i) or (ii). Without loss of generality $|C_1|\geq |C_2|\geq\dots\geq |C_t|$. If $|C_2|+|C_3|+\dots+|C_t|\geq m$, then we get $Conn(G\setminus S)=|C_1|\leq |G|-m$ and so (i) holds. 
Therefore, we can suppose that  $|C_2|+|C_3|+\dots+|C_t|< m$ holds. In particular this means that $t\leq m-1$, and hence part (b) of Lemma~\ref{LemmaManyComponents} holds. Let $S'=S\cup C_2\cup\dots\cup C_t$ to get a set with $|S'|\leq |S|+|C_2|+|C_3|+\dots+|C_t|\leq km$ such that $G\setminus S'=C_1$ which is $k$-connected.
\end{proof}

We combine Lemma~\ref{Disconnecting} with Lemma~\ref{4partite} in order to classify colourings of $K_n$ which do not have a large monochromatic highly connected component.
\begin{lemma}\label{CasesLemma}
 Suppose that the edges of $K_n$ are coloured with $3$ colours.  There is a set of vertices $S\subseteq V(K_n)$ of order less than $\cCases$ such that $K_n\setminus S$ satisfies one of the following.
 \begin{enumerate}[\normalfont(i)]
  \item $K_n\setminus S$ is either empty or $\kCases$-connected in some colour.
  \item $K_n\setminus S$ is $4$-partite such that every class of the 4-partition has at least $3$ vertices in it.
  \item The vertices of $K_n\setminus S$ can be partitioned into four sets $\Ag$, $\Ar$, $\Ab$, and $W$ with the following properties.  
  \begin{itemize}
   \item $\Bipart{K_n}{\Ab}{\Ar}$ is green, $\Bipart{K_n}{\Ag}{\Ar}$ is blue, and $\Bipart{K_n}{\Ag} {\Ab}$ is red.
   \item The edges in $\Bipart{K_n}{\Ag}{W}$ are red or blue. \\The edges in $\Bipart{K_n}{\Ab}{W}$ are red or green.  \\The edges in $\Bipart{K_n}{\Ar}{W}$ are blue or green.
  \item $|\Ag|, |\Ar|, |\Ab| \geq \mCases$.
  \end{itemize}

 \end{enumerate}
\end{lemma}
\begin{proof}

Let $c_g=\cDisconnecting(\kCases,402)$, 
$c_b=\cDisconnecting(\kCases, c_g+402)$, and
$c_r=\cDisconnecting(\kCases, c_b+c_g+402)$, where $\cDisconnecting(*,*)$ is the function in Lemma~\ref{Disconnecting}.  Notice that $c_g+c_b+c_r\leq 41000$.
Suppose that (i) does not hold in any colour.

Since (i) does not hold in red, Lemma~\ref{Disconnecting} implies that there is a set $S_1$ with $|S_1|\leq c_r$ such that $K_n\setminus S_1$ has no red components of order more that $|K_n\setminus S_1|- c_b-c_g-402$.

Since (i) does not hold in blue, we can apply Lemma~\ref{Disconnecting} to the blue colour class of $K_n\setminus S_1$ to obtain a set $S_2$ with $|S_2|\leq c_b$ such that $K_n\setminus (S_1\cup S_2)$ has no red or blue components of order more than $|K_n\setminus (S_1\cup S_2)|-c_g-402$.

Since (i) does not hold in green, we can apply Lemma~\ref{Disconnecting} to the green colour class of $K_n\setminus S_1$ to obtain a set $S_3$ with $|S_3|\leq c_g$ such that $K_n\setminus (S_1\cup S_2\cup S_3)$ has no red, blue, or green components of order more than $|K_n\setminus (S_1\cup S_2\cup S_3)|-402$.

Let $S=S_1\cup S_2\cup S_3$ to get a set with $|S|\leq 41000$.  We have that $K_n\setminus S$ has no monochromatic components of order more than $|K_n\setminus S|-402$. In particular this means that $K_n\setminus S$ is disconnected in each colour. 
\begin{claim}\label{CasesClaim}
Suppose that $R$ and $B$ are red and blue components of $K_n\setminus S$ respectively with $R\not\subseteq B$ and $B\not\subseteq R$.  If $K_n\setminus (S\cup R\cup B)\leq 2$, then there is a partition of $K_n$ satisfying part (iii) or the lemma. 
\end{claim}
\begin{proof}
Let $S'=K_n\setminus (R\cup B\cup G)$. By the assumption of the claim $|S'|\leq |S|+2\leq \cCases$.
Since $G\setminus S$ had no monochromatic components of order more than $|K_n\setminus S'|-402$, we get that  $G\setminus S'$ has no monochromatic components of order more that $|K_n\setminus S'|-\mCases$.
 Notice that $B\setminus R$ and $R\setminus B$ are non-empty and all the edges between these two sets are green.
Therefore, there is a green component, $G$, of $K_n\setminus S'$ containing $B\setminus R$ and $R\setminus B$.     
Let $\Ar=B\cap G\setminus R, \Ab=R\cap G\setminus B$, and $\Ag=R\cap B\setminus G$ and $W=B\cap R\cap G$.

We have that $\Ar\cap R=\emptyset$, so there are no red edges between $\Ar$ and $W$, and so all the edges in $\Bipart{K_n}{\Ar} {W}$ are blue or green. Similarly, we get that there are no red edges between $\Ar$ and $\Ab$ (since $\Ar\cap R=\emptyset$), and no blue edges between $\Ar$ and $\Ab$ (since $\Ab\cap B=\emptyset$).  This means that $\Bipart{K_n}{\Ab} {\Ar}$ is green.  By the same argument, it is easy to see that the other colours between pairs of $\Ab$, $\Ar$, $\Ag$, and $W$ are as specified in case (iii) of the lemma.

Notice that from the definition of $G$ and $S'$ we have $(B\cap G)\cup R=(G\cap R)\cup B=(R\cap B)\cup G=K_n\setminus S'$.
Combining $(B\cap G)\cup R=K_n\setminus S'$ with the fact that $|R|\leq K_n\setminus S'-\mCases$ we get $|\Ar|\geq |B\cap G|-|R|\geq |K_n\setminus S'|-|R|\geq \mCases$. Similarly we get $|\Ab|\geq \mCases$ and $|\Ag|\geq \mCases$, proving the claim.
\end{proof}

Suppose that the colouring on $K_n\setminus S$ is 4-partite.  Let $A_1$, $A_2$, $A_3$, and $A_4$ be the classes of the 4-partition of $K_n\setminus S$.
Either part (ii) of the lemma holds, or $|A_i|\leq 2$ for some $i$.  Without loss of generality, we may suppose that $i=4$.  In this case the red component $A_1\cup A_2$ and the blue component $A_2\cup A_3$ satisfy the conditions of Claim~\ref{CasesClaim} which implies that there is a partition of $K_n\setminus S$ as in part (iii) of the lemma.

Suppose that the colouring on $K_n\setminus S$ is not 4-partite.
Let $R$ be the largest component of $K_n\setminus S$ in any colour.  Without loss of generality, we may suppose that $R$ is red.  Let $B$ be a blue component which is not contained in $R$. Let $A_1=R\cap B$, $A_2=R\setminus B$,  $A_3=B\setminus R$, and $A_4=K_n\setminus (S\cup A\cup B)$.  Notice that $A_3$ is non-empty since $B\not\subseteq R$, $A_2$  is non-empty since $|R|\geq |B|$, and $A_1$ is nonempty, since otherwise $\Bipart{K_n}{R}{B}=\Bipart{K_n}{A_1}{A_2}$ would be a green complete bipartite graph, contradicting $R$ being the largest component.  Therefore, by Lemma~\ref{4partite}, since $K_n\setminus S$ is not 4-partite, we have that $|A_4|=\emptyset$.  We can apply Claim~\ref{CasesClaim} in order to find a partition as in part (iii) of the lemma.
\end{proof}

In order to prove Theorem~\ref{ThreeCycles}, we will need a number of results about partitioning coloured graphs into monochromatic subgraphs.  One of these is the 3-colour case of a weakening of Conjecture~\ref{Erdos} when ``cycles'' is replaced by ``paths''.

\begin{theorem} [\cite{PokrovskiyCycles}]\label{ThreePaths}
 Suppose that the edges of $K_n$ are coloured with $3$ colours.  Then $K_n$ can be partitioned into $3$ monochromatic paths.
\end{theorem}

Theorem~\ref{ThreePaths} was proved by the author in \cite{PokrovskiyCycles}.  
We'll need a lemma about partitioning a $2$-edge-coloured $K_n$ into a path and a complete bipartite graph. This lemma was first proved by Ben-Eliezer,  Krivelevich, and Sudakov in \cite{BKS} (see Lemma 4.4  and its proof.)
It later appeared in the form in which we use it in \cite{PokrovskiyCycles}.
\begin{lemma}[\cite{BKS}]\label{PathBipartite}
 Suppose that the edges of $K_n$ are coloured with two colours.  Then $K_n$ can be covered by a red path and a disjoint blue balanced complete bipartite graph.
\end{lemma}
The following is a corollary of the above lemma.
\begin{lemma}\label{BipartitePathBipartite}
Suppose that $K_{n,n}$ is coloured with $2$ colours.
Then $K_{n,n}$ can be partitioned into one red path and two blue balanced complete bipartite graphs.
\end{lemma}
\begin{proof}
Let $A$ and $B$ be the parts of $K_{n,n}$. Considering non-edges of $K_{n,n}$ as blue edges, apply Lemma~\ref{PathBipartite} to the graph to obtain a partition of $K_{n,n}$ into a red path $P$ and two sets $X$ and $Y$ such that $|X|=|Y|$ and there are no red edges between $X$ and $Y$.

Notice that since $|K_{n,n}\setminus P|=|X| + |Y|$ is even and the vertices of $P$ alternate between $A$ and $B$ and  we have $|A\setminus  P|= |B\setminus P|$.
Combining $|X\cap A|+|X\cap B|=|X|=|Y|=|Y\cap A|+|Y\cap B|$ and $|X\cap A|+|Y\cap A|=|A\setminus  P|= |B\setminus P|=|X\cap B|+|Y\cap B|$ implies that we have $|X\cap A|=|Y\cap B|$ and  $|X\cap B|=|Y\cap A|$.
Notice that since all the edges between $X\cap A$ and $Y\cap B$ are present in $G$, and none of them are red, we have that the subgraph of $K_{n,n}$ on $(X\cap A)\cup (Y\cap B)$ is a blue complete bipartite graph. Since $|X\cap A|=|Y\cap B|$, this complete bipartite graph is balanced. For the same reason, the subgraph on $K_{n,n}$ on $(Y\cap A)\cup (X\cap B)$ is a blue balanced complete bipartite graphs. Thus $(X\cap A)\cup (Y\cap B)$, $(Y\cap A)\cup (X\cap B)$, and $P$, give the required partition of $K_{n,n}$. 
\end{proof}

We use Lemma~\ref{BipartitePathBipartite} to prove the following.

\begin{lemma}\label{BipartitePathsAndCycle}
 Suppose that $K_{n,n}$ is coloured with $2$ colours.
Then $K_{n,n}$ can be partitioned into two red paths and one blue cycle.
 \end{lemma}
 \begin{proof}
 Let $X_1$ and $X_2$ be the parts of the bipartition of $K_{n,n}$.
By Lemma~\ref{BipartitePathBipartite}, $K_{n,n}$ can be partitioned into a blue path $Q$ and two red paths $P_1$ and $P_2$.  In addition, we may suppose that $|Q|$ is as small as possible in such a partition.  Let $q_1$ and $q_2$ be  the two endpoints of $Q$. 

Suppose that $q_1$ and $q_2$ are both in $X_1$. Then there must be an endpoint $p$, of either $P_1$ or $P_2$ which is in $X_2$. 
Notice that the edges $pq_1$ and $pq_2$ must both be blue (since otherwise we could join $q_1$ or $q_2$ to the path containing $p$ , contradicting the minimality of $|Q|$). 
This allows us to construct a blue cycle by adding the edges $q_1p$ and $q_2p$ to $Q$, which together with $P_1-p$ and $P_2-p$ gives the required partition of $K_{n,n}$.


The same argument works if $q_1$ and $q_2$ are both in $X_2$. Therefore, we can assume that $q_1\in X_1$ and $q_2\in X_2$. 
This means that out of the four endpoints of $P_1$ and $P_2$, two are in $X_1$ and two are in $X_2$. Without loss of generality, we may assume that an endpoint $p_1$ of $P_1$ is in $X_1$, and an endpoint $p_2$ of $P_2$ is in $X_2$.
As before, the edges $p_1q_1$ and $p_2q_2$ must be blue (since otherwise for some $i$ we could extend $P_i$ by adding the vertex $q_i$, contradicting the minimality of $Q$).  

Suppose that the edge $p_1p_2$ is red. Then we can join $P_1$ and $P_2$ with this edge to obtain a red path spanning $P_1\cup P_2$.  This gives us a partition into two red paths $P_1\cup P_2$ and $\{q_2\}$ and a blue path $Q-q_2$.
However $|Q-q_2|=|Q|-1$ contradicting minimality of $|Q|$ in the original partition.

Therefore, we can suppose that the edge $p_1p_2$ is blue.  Notice that adding the edges $q_1p_1$, $p_1p_2$, and $p_2q_2$ to $Q$ produces a cycle.  This gives us a partition of $K_{n,n}$ into two red paths $P_1-p_1$, $P_2-p_2$ and a blue cycle $Q+p_1+p_2$.
\end{proof}

Combining the above lemma with Lemma~\ref{PathBipartite} gives us the following.
\begin{lemma}\label{PathsAndCycle}
 Suppose that $K_n$ is coloured with $3$ colours.  Then $K_n$ can be partitioned into one red path, two blue paths and one green cycle.
\end{lemma}
\begin{proof}
First apply Lemma~\ref{PathBipartite} in order to partition $K_n$ into a red path and a  blue-green balanced complete bipartite graph.  Then apply Lemma~\ref{BipartitePathsAndCycle} to this graph.
\end{proof}

Lemma~\ref{BipartitePathBipartite} also allows us to bound the number of monochromatic paths needed to partition a $3$-coloured complete bipartite graph.

\begin{lemma}\label{ThreeColourBipartite}
 Suppose that $K_{n,n+t}$ is coloured with $3$ colours.  Then $K_{n,n+t}$ can be partitioned into one red path, two blue paths, four green paths and $t$ vertices.
\end{lemma}
\begin{proof}
It is sufficient to prove the lemma for $t=0$, since by choosing the $t$ vertices to be any vertices in the larger partition of $K_{n,n+t}$ we reduce to the $t=0$ case.
Apply Lemma~\ref{BipartitePathBipartite} in order to partition $K_{n,n}$ into a red path and two balanced blue-green complete bipartite graphs.  Then apply  Lemma~\ref{BipartitePathBipartite} to each of these graphs.
\end{proof}

The following lemma is used to cover $K_{n,n}$ with very few red edges.
\begin{lemma}\label{LemmaPartitionStarComplement}
Let $K_{n,n}$ be $2$-edge-coloured such that the red colour class is a union of stars. Then $K_{n,n}$ can be covered by three disjoint blue paths.
\end{lemma}
\begin{proof}
By Lemma~\ref{BipartitePathBipartite}, we can cover $K_{n,n}$ by a blue path $P$ and two red balanced complete bipartite graphs $H$ and $J$. Since the red colour class is a union of stars we have that $|H|, |J|\in\{0,2\}$. If $H$ or $J$ is empty then we have a partition of $K_{n,n}$ into the blue path $P$ and two vertices. Otherwise if $|H|,|J|=2$, then the edges between $H$ and $J$ must be blue (since the red colour class is a union of stars). This gives us a partition of $K_{n,n}$ into the blue path $P$ and two blue edges.
\end{proof}
We remark that it is easy to prove a stronger version of above lemma with just two blue paths covering, though we won't need it.

Finally we need the following lemma about covering one part of a 3-coloured complete bipartite graph.
\begin{lemma}\label{PathsAndConnected}
 Suppose that the complete bipartite graph $K$ with parts $X$ and $Y$ satisfying $|X|\leq |Y|$ is coloured with $3$ colours.  Then $X$ can be covered with $7$ disjoint monochromatic paths $P_1, \dots, P_{7}$ such that either every vertex in $Y\setminus (P_1\cup \dots \cup P_{7})$ has two {red} neighbours in $X$ or the paths $P_1,\dots, P_{7}$ are all blue or green.
\end{lemma}
\begin{proof}
Let $S$ be the set of vertices in $Y$ which have at most one {red} edge to $X$.  

If $|S|\geq |X|$ holds, then by Lemma~\ref{BipartitePathBipartite}, it is possible to cover $X$, and $|X|$ vertices in $S$ by a green path $Q$ and two {red}-blue balanced complete bipartite graphs $H$ and $J$.  Since vertices in $S$ contain at most one {red} edge, the {red} subgraphs on $H$ and $J$ are disjoint unions of stars.  By Lemma~\ref{LemmaPartitionStarComplement} we can cover each of  $H$ and $J$ by three blue paths, which together with $Q$ give the required covering of $X$.

If $|S|\leq |X|$, then by Lemma~\ref{ThreeColourBipartite}, it is possible to cover $X$, $S$, and $|X|-|S|$ vertices in $Y$ by $7$ disjoint monochromatic paths. The uncovered vertices in $Y$ are all outside $S$, and so each have at least $2$ {red} neighbours in $X$.
\end{proof}

We are now ready to prove Theorem~\ref{ThreeCycles}.
\begin{proof}[Proof of Theorem~\ref{ThreeCycles}]
By Lemma~\ref{CasesLemma} $K_n$ has a set of vertices $S$ with $|S|\leq \cCases$ such that $K_n\setminus S$ has one of the structures (i) -- (iii) in Lemma~\ref{CasesLemma}.
There are three cases, depending on which structure from  Lemma~\ref{CasesLemma}  we have.

Suppose that Case (i) of Lemma~\ref{CasesLemma} occurs.  Without loss of generality, we can suppose that $K_n\setminus S$ is $4$-connected in red.
Apply Theorem~\ref{CycleMinDegreeConn} to the red subgraph of $K_n$ on $K_n\setminus S$ in order to partition $K_n\setminus S$ into a red cycle $C$ and a blue-green graph $H$ satisfying $\delta(H)\geq \frac{4}{5}|H|-3$.  If $\delta(H)<\frac{3}{4}|H|$, then we have $|H|\leq 60$ and so we are done since we can leave all the vertices of $H$ uncovered.  Otherwise, $\delta(H)\geq \frac{3}{4}|H|$ we can apply Theorem~\ref{Letzter} to $H$ in order to partition $H$ into a blue cycle a green cycle, which together with $C$ cover $\geq n-\cCycles$ vertices.

Suppose that Case (ii) of Lemma~\ref{CasesLemma} occurs.
Let $A_1$, $A_2$, $A_3$, and $A_4$ be the classes of the 4-partition of $K_n\setminus S$.  Since $|A_i|\geq 3$, we can let $A'_i$ be a set of $3$ vertices from $A_i$ for each $i$.  Let $K$ be the complete graph with vertices $K_n\setminus (S\cup A'_1\cup A'_2\cup A'_3\cup A'_4)$.  By Theorem~\ref{ThreePaths}, $K$ has a partition into three disjoint monochromatic paths.  Since the colouring on $K_n\setminus S$ was 4-partite, these monochromatic paths must be contained in $A_i\cup A_j$ for some $i$ and $j$.  It is easy to see that any such path can be turned into a cycle by adding at most one vertex from $A'_i$ and one vertex from $A'_j$.  Therefore, all three paths can be closed by using distinct vertices from the sets $A'_i$ for various $i$.  All together we have left at most $|S|+12\leq \cCycles$ vertices uncovered.

Suppose that Case (iii) of Lemma~\ref{CasesLemma} occurs.  Let $W$, $\Ar$, $\Ab$, and $\Ag$ be as in Case (iii) of Lemma~\ref{CasesLemma}.  
Partition $\Ag$ arbitrarily into a set of $\DSize$ vertices $\Dg$ and a set $\Apg=\Ag\setminus \Dg$.  Do similarly to $\Ab$ and $\Ar$ to obtain $\Db, \Dr, \Apb$, and $\Apr$ with $|\Db|=|\Dr|=400$.
Let $D= \Dg\cup \Db\cup \Dr$ and $A'= \Apg\cup \Apb\cup \Apr$.

We can partition $W$ into three sets $\Cg, \Cr,$ and $\Cb$ such that for every $v \in \Cg$ there is at least $200$ red edges and $200$ blue edges between $v$ and $D$, and similarly for $\Cr,$ and $\Cb$.  This is possible since by definition of the sets $A_{*,*}$, for any vertex $v\in W$, no three edges between $v$ and $\Ag$,  $\Ab$, and $\Ar$ can all be of the same colour.

We'll need the following definition.
\begin{definition}
 We say that a red path with vertex sequence $p_1, p_2, \dots, p_{i-1}, p_{i}$ is \textbf{linkable} if $|\{p_1, p_2\}\cap (\Apg\cup \Apb \cup \Cg\cup \Cb)|\geq 1$ and  $|\{p_{i-1}, p_i\}\cap (\Apg\cup \Apb \cup \Cg\cup \Cb)|\geq 1$.
\end{definition}
Blue and green linkable paths are defined similarly (replacing the set $\Apg\cup \Apb \cup \Cg\cup \Cb$ in the definition with  $\Apr\cup \Apg \cup \Cr\cup \Cg$ for blue-linkable paths and by $\Apr\cup \Apb \cup \Cr\cup \Cb$ for green linkable paths).

Notice that every path which doesn't use edges inside $\Apg$,  $\Apb$, $\Apr$, $\Cg$,  $\Cb$, or $\Cr$ must be linkable. Indeed, suppose that  $p_1, p_2, \dots, p_{i-1}, p_{i}$ is a red  path with no edges inside $\Apg$,  $\Apb$, $\Apr$, $\Cg$,  $\Cb$, or $\Cr$. Then $p_1$ and $p_2$ must be in distinct sets from  $\{\Apg$,  $\Apb$, $\Apr$, $\Cg$,  $\Cb$,  $\Cr\}$. Notice that we cannot have $p_1,p_2\in \Apr\cup \Cr$ since all the edges between $\Apr$ and $\Cr$ are blue or green. Therefore either $p_1$ or $p_2$ must be contained in $\Apg\cup \Apb \cup \Cg\cup \Cb$. By the same argument we have that  $p_{i-1}$ or $p_i$ is  contained in $\Apg\cup \Apb \cup \Cg\cup \Cb$, which shows that the path is linkable. In particular, we have shown that single vertex paths are linkable.

The following is the main property that linkable paths have.
\begin{claim}\label{LinkPaths}
Let $\Dpg\subseteq \Dg$ and $\Dpb\subseteq\Db$ with $|\Dpg|, |\Dpb|\geq 250$.
 Suppose that we have disjoint red linkable paths $P_1,\dots, P_t$ with $t\leq 21$.  There is a red cycle contained in $P_1\cup\dots\cup P_t\cup\Dpg\cup\Dpb$  which consists of at most $3t$ vertices in $\Dpg\cup\Dpb$ and all, except at most, $2t$ vertices in $P_1\cup\dots\cup P_t$.
\end{claim}
\begin{proof}
Recall that every vertex in $\Apg\cup \Apb \cup \Cg\cup \Cb$ has at least $200$ red edges going to $\Dg\cup \Db$ (For $\Cg$ and $\Cb$ this holds by the definition of these sets.  For $\Apg$ and $\Apb$ it holds since from Lemma~\ref{CasesLemma} we have only red edges between $\Ag$ and $\Ab$.) Since $|\Dpg|, |\Dpb|\geq 250$ and $|\Dg|= |\Db|=400$ we get that every vertex in $\Apg\cup \Apb \cup \Cg\cup \Cb$  has at least $50$ red edges going to $\Dpg\cup \Dpb$.

 Let the vertex sequence of $P_1$ be $p_1, p_2, \dots, p_{i-1}, p_i$.  Notice that by definition of ``red linkable'' one of the vertices $p_1$ or $p_2$ must lie in one of the sets $\Apg, \Apb, \Cg,$ or $\Cb$.  Therefore either $p_1$ or $p_2$ is joined to at least $50$ vertices in $\Dpg\cup\Dpb$ by a red edge.  Similarly either $p_{i-1}$ or $p_i$ is joined to at least $50$ vertices in $\Dpg\cup\Dpb$ by a red edge.  Possibly removing $p_1$ and/or $p_i$ we can replace $P_1$ by a red path $P'_1$ which starts and ends in $\Dpg\cup\Dpb$ and passes through $p_2, \dots, p_{k-2}$.  
 
 We can do the same with $P_2, \dots, P_t$ to obtain paths $P'_2, \dots, P'_t$ with endpoints in $\Db\cup \Dg$.  Notice that these paths can be chosen to be disjoint---using the fact that $t\leq 21$ and  every vertex in any of $\Apg, \Apb, \Cg,$ or $\Cb$ has at least $50$ red neighbours in $\Dpg\cup\Dpb$.   Therefore, since $\Bipart{K_n}{\Dpb}{\Dpg}$ is a red complete bipartite graph with  $\geq 250$ vertices in every class, we can join all the paths $P'_1, \dots, P'_t$ together, using at most $t$ extra vertices in $\Dpg\cup\Dpb$ to form the required red cycle.
\end{proof}
It is easy to see that similar claims hold for blue and green linkable paths as well. We can use these to prove the following lemma.
\begin{claim}
Suppose that one of the following holds.
\begin{enumerate}[(1)]\label{LinkablePathsToPartition}
\item  $K_n\setminus (D\cup S)$ can be partitioned into  $\leq 63$ disjoint monochromatic linkable paths with at most $21$ of each colour.
\item  $K_n\setminus (D\cup S)$ can be partitioned into into a red cycle, $t\leq 21$ disjoint blue linkable paths and $s\leq 21$ disjoint green linkable paths.
\end{enumerate}
Then there is a cover of at least $n-\cCycles$ of the vertices of $K_n$ by three disjoint monochromatic cycles.
\end{claim}
\begin{proof}
For (1), let $R_1, \dots, R_r$ be red linkable paths, $B_1, \dots, B_b$ be blue linkable paths, and $G_1, \dots, G_g$ be green linkable paths such that $r, b, g\leq 21$ and $K_n\setminus (D\cup S)=R_1\cup \dots\cup R_r\cup B_1\cup \dots\cup B_b \cup G_1\cup \dots\cup G_g$. By Claim~\ref{LinkPaths}, applied to the red paths with $\Dpb=\Db$ and $\Dpg=\Dg$, there is a red cycle $R$ contained in $R_1\cup \dots\cup R_r\cup D$ covering all except $2r$ vertices in $R_1, \dots, R_r$, and at most $3r$ vertices in $\Db\cup \Dg$. Similarly applying Claim~\ref{LinkPaths}, to the blue paths with $\Dpg=\Dg\setminus R$ and $\Dpr=\Dr$, we get a blue cycle $B$ contained in $B_1\cup \dots\cup B_r\cup D\setminus R$ covering all except $2b$ vertices in $B_1, \dots, B_b$, and at most $3b$ vertices in $\Dr\cup \Dg\setminus R$.  Finally applying Claim~\ref{LinkPaths}, to the green paths with $\Dpb=\Db\setminus R$ and $\Dpr=\Dr\setminus B$, we get a green cycle $G$ contained in $G_1\cup \dots\cup G_r\cup D\setminus (R\cup G)$ covering all except $2g$ vertices in $G_1, \dots, G_g$, and at most $3g$ vertices in $D\setminus (R\cup G)$. By construction $R$, $B$, and $G$  are disjoint cycles covering all, except at most $2(r+b+g)+|D|+|S|\leq \cCycles$ vertices in $K_n$.

The proof of part (2) is identical, except that we can skip the first application of Claim~\ref{LinkPaths} where we linked the paths $R_1, \dots, R_r$ together.
\end{proof}

We now complete the proof of the theorem.
There are two cases depending on whether $A'$ or $W$ is larger.

\begin{proof}[Suppose that $|W|\geq |A'|$]

Without loss of generality we may assume that $|\Cr|\geq |\Cb|\geq |\Cg|$.  There are three cases depending on whether $|A'|\geq |\Cr|-|\Cb|-|\Cg|$ and $|\Cr|\geq|\Cb|+|\Cg|$ hold.

\textbf{Case 1:}  Suppose that we have $|A'|\geq |\Cr|-|\Cb|-|\Cg|\geq 0$.

Notice that $|\Cr|-|\Cb|-|\Cg|\geq 0$ implies  $\frac{1}{2}(|A'|+|\Cr|-|\Cb|-|\Cg|)\geq 0$ and  $|W|\geq |A'|$ is equivalent to $|\Cr|\geq \frac{1}{2}(|A'|+|\Cr|-|\Cb|-|\Cg|)$.
Therefore we can let $\Er$ be a subset of $\Cr$ of order $\left\lfloor\frac{1}{2}(|A'|+|\Cr|-|\Cb|-|\Cg|)\right\rfloor$.  Similarly, notice that $|W|\geq |A'|$ is equivalent to $|\Cb|+ |\Cg| \geq \frac{1}{2}(|A'|-|\Cr|+|\Cb|+|\Cg|)$ and $|A'|\geq |\Cr|-|\Cb|-|\Cg|$ is equivalent to $\frac{1}{2}(|A'|-|\Cr|+|\Cb|+|\Cg|)\geq 0$. Therefore, it is possible to choose  $\Eb$ and $\Eg$ subsets of $\Cb$ and $\Cg$ respectively satisfying $|\Eb|+ |\Eg| = \left\lfloor\frac{1}{2}(|A'|-|\Cr|+|\Cb|+|\Cg|)\right\rfloor$.  

Notice that we have $\big||A'|-  |\Er|-|\Eb|-|\Eg|\big|\leq 2  $ and also $\big||\Cr\setminus \Er| - |(\Cb\cup \Cg)\setminus(\Eb\cup \Eg)|\big|\leq 2$.  This, together with Lemma \ref{ThreeColourBipartite} means that $\Bipart{K_n}{A'}{\Er\cup\Eb\cup\Eg}$ and $\Bipart{K_n}{\Cr\setminus \Er}{(\Cb\cup \Cg)\setminus(\Eb\cup \Eg)}$ can each be partitioned into $7$ monochromatic paths and $2$ extra vertices.  Each of these paths must be linkable (since they do not use edges within the sets $A$, $\Cg, \Cr,$ or $\Cb$).  Therefore, by part (1) of Claim~\ref{LinkablePathsToPartition}, there are three monochromatic cycles passing through all but $\cCycles$ vertices of $K_n$.

\textbf{Case 2:}  Suppose that we have $|\Cr|-|\Cb|-|\Cg|\geq |A'|$. Since this is equivalent to $|\Cr|\geq |\Cb\cup \Cg\cup A'|$, we can  apply Lemma~\ref{PathsAndConnected} to $\Bipart{K_n}{\Cr}{A'\cup\Cb\cup\Cg}$ in order to obtain $7$ monochromatic paths $P_1, \dots, P_{7}$ covering $A'\cup\Cb\cup\Cg$ and $|A'\cup\Cb\cup\Cg|$ vertices in $\Cr$.

If $P_1, \dots, P_{7}$ are all blue or green, then by Lemma~\ref{PathsAndCycle}, we can cover $\Cr\setminus (P_1\cup\dots\cup P_7)$ with one blue monochromatic path $P_{8}$, two green monochromatic paths $P_{9}, P_{10}$, and a red monochromatic cycle $C_1$.
It is easy to see that the paths $P_1, \dots, P_{10}$ are all linkable, so using part (2) of Claim~\ref{LinkablePathsToPartition}, we can partition all but $\cCycles$ vertices of $K_n$ into three monochromatic cycles.

If any of the paths $P_1, \dots, P_{7}$ are red, Lemma~\ref{PathsAndConnected} would have implied that every vertex in $\Cr\setminus P_1\cup \dots\cup P_{7}$ must have two red neighbours in $A'\cup\Cb\cup\Cg$. By Lemma~\ref{PathsAndCycle} we can partition  $\Cr\setminus (P_1\cup\dots\cup P_7)$ into one blue monochromatic path $P_{8}$, two green monochromatic paths $P_{9}, P_{10}$, and a red monochromatic path $R$.  Let $r_1$ and $r_2$ be the endpoints of $R$.  Let $x_1$ and $x_2$ be two distinct red neighbours of $r_1$ and $r_2$ respectively in $A'\cup\Cb\cup\Cg$.  The red path $x_1+R+ x_2$ is linkable.  Since we only removed two vertices from $P_1\cup \dots\cup P_{7}$, the set $(P_1\cup \dots\cup P_{7})\setminus\{x_1, x_2\}$ must be spanned by $9$ paths which are all linkable.
Therefore, by part (1) of Claim~\ref{LinkablePathsToPartition}, there are three monochromatic cycles passing through all but $\cCycles$ vertices of $K_n$

\textbf{Case 3:}  Suppose that we have $|\Cr|\leq|\Cb|+|\Cg|$.
Notice that we also have $|W\setminus (\Cg\cup \Cr\cup \Cb)|\leq |A'|$ and $|\Cr|\geq |\Cb|, |\Cg|$.

We choose three subsets $\Cpg, \Cpr,$ and $\Cpb$ of $\Cg, \Cr,$ and $\Cb$ respectively such that  $|\Cpr|\leq |\Cpb|+ |\Cpg|$,  $|W\setminus (\Cpg\cup \Cpr\cup \Cpb)|\leq |A'|$, and $|\Cpr|\geq |\Cpb|, |\Cpg|$ all hold. In addition we make $|\Cpg|+ |\Cpr|+|\Cpb|$ as small as possible.
Notice that this ensures that $|W\setminus (\Cpg\cup \Cpr\cup \Cpb)|\geq |A'|-2$ holds (since otherwise we could remove a vertex from $\Cpr$ and one or both of $\Cpb$ and $\Cpg$ in order to obtain a triple satisfying the three inequalities with smaller $|\Cpg|+ |\Cpr|+|\Cpb|$).

Notice that $|\Cpr|\geq |\Cpb|, |\Cpg|$ implies that $|\Cpg|,|\Cpb|\geq \frac{1}{2}(|\Cpg|+|\Cpb|-|\Cpr|)$. We also have $|\Cpg|+|\Cpb|-|\Cpr|\geq 0$.
Therefore we can choose $\Eb$ and $\Eg$ subsets of $\Cpb$ and $\Cpg$ respectively such that $|\Eb| =|\Eg|= \left\lfloor\frac{1}{2}(|\Cpg|+|\Cpb|-|\Cpr|)\right\rfloor$  holds.

Notice that we have $\big||\Cpr|- |(\Cpg\cup \Cpb)\setminus (\Eg\cup \Eg)|\big|\leq 2$, $|\Eb|=|\Eg|$, and $\big||A'|- |W\setminus (\Cpg\cup \Cpr\cup \Cpb)|\big|\leq 3$. Therefore, by Lemma~\ref{ThreeColourBipartite} we can cover each of  $B(\Cpr, (\Cpg\cup \Cpb)\setminus (\Eg\cup \Eg))$ , $B(\Eb, \Eg)$, and $B(A', W\setminus (\Cpg\cup \Cpr\cup \Cpb))$ by $7$ monochromatic paths, plus at most $5$ isolated vertices.
 Each of these paths must be linkable (since they have no edges inside the sets $\Apg$,  $\Apb$, $\Apr$, $\Cg$,  $\Cb$, and $\Cr$).  Therefore, by Claim~\ref{LinkPaths}, there are three monochromatic cycles passing through all but $\cCycles$ vertices of $K_n$.
\end{proof}
 
 The case when $|A'|\geq |W|$ holds is proved identically, exchanging the roles of the sets $A'_{c,d}$ and $W_{c,d}$ for every pair of colours $(c,d)$.  To see that we can do this, first notice that the only difference between the sets $A'_{*,*}$ and $W_{*,*}$ is that we know that the edges between $A'_{c,d}$ and $A'_{c',d'}$ are always the same colour for particular pairs of colours $(c,d)$ and $(c',d')$ (whereas the edges between $W'_{c,d}$ and $W'_{c',d'}$  can be coloured arbitrarily).  However, looking through the proof of the case when $|W|\geq|A'|$ it is easy to check that edges contained within the set $A'$ were never used to construct the cycles covering $K_n$.  Therefore, the same proof still works after exchanging the roles of the sets $A'_{c,d}$ and $W_{c,d}$.
\end{proof}

\section{Concluding remarks}\label{SectionConclusion}
Here we make some remarks about possible further directions for research hin this  new area.

\subsubsection*{Optimal bound in Theorem~\ref{CycleMinDegreeConn}}
As shown in Figure~\ref{FigureKConnectedExtremal}, the constant $\frac{1}{k+1}$ in front of $|H|$ in (\ref{EqCycleMinDegreeConn}) cannot be decreased. However the ``$+3$'' term can certainly be improved. The following problem seems natural.
\begin{problem}\label{ProblemOptimizeConnectedTheorem}
For $k\geq 1$, determine the smallest number $c_k$ such that every $k$-connected graph $G$ has a cycle $C$ with 
$$\Delta(G[V(G)\setminus C])\leq \frac{1}{k+1}|V(G)\setminus C|+c_k.$$ 
\end{problem}
Note that from Theorem~\ref{CycleMinDegree} we know the solution to the above problem for $k=1$---specifically we have $c_1=-\frac{1}2$.  This shows that the constants $c_k$ in problem~\ref{ProblemOptimizeConnectedTheorem} may be negative. The construction showing the lower bound of $c_1\geq-\frac{1}2$ is the graph in Figure~\ref{FigureGeneralGraphExtremal}. The fact that this graph is fairly complicated, and has no obvious generalization to higher $k$ suggests that Problem~\ref{ProblemOptimizeConnectedTheorem} may be difficult.

For $k\geq 2$, the best bounds we currently know for Problem~\ref{ProblemOptimizeConnectedTheorem} are $-\frac{2k+1}{k+1}\leq c_k\leq 3-\frac{4}{1+k}$. The lower bound comes from Figure~\ref{FigureKConnectedExtremal}, and the upper bound comes from the proof of Theorem~\ref{CycleMinDegreeConn}.

\subsubsection*{Approximate versions of the EGP Conjecture}
In this paper we proved a new approximate version of Conjecture~\ref{Erdos} for $r=3$.  In~\cite{PokrovskiyCycles}, the author conjectured that a similar approximate version holds for any number of colours.
\begin{conjecture}[\cite{PokrovskiyCycles}]\label{approximate}
For each $r$ there is a constant $c_r$, such that in every $r$-edge coloured complete graph $K_n$, there are $r$ vertex-disjoint monochromatic cycles covering $n - c_r$ vertices in $K_n$.  
\end{conjecture}
A weaker, but still interesting version of this conjecture, arises if we replace $c_r$ with a function $o_r(n)$ which statisfies $\frac{o_r(n)}{n}\to 0$ as $n\to \infty$.

The constant $c_1\geq \cCycles$ in Theorem~\ref{ThreeCycles} could certainly be improved by being more careful throughout the proof of the theorem.  As mentioned in the introduction, Letzter~\cite{Letzter3Cycles} has an alternative proof of Theorem~\ref{ThreeCycles} with a better constant of $60$ instead of $\cCycles$. The only lower bound on this constant is ``$c_1\geq 1$'' which comes from the counterexamples to Conjecture~\ref{Erdos} constructed in \cite{PokrovskiyCycles}.  We conjecture that this lower bound is correct.
\begin{conjecture}\label{ThreeCyclesConjecture}
Every $3$-edge-coloured complete graph contains $3$ disjoint monochromatic cycles which cover all except possibly one vertex.
\end{conjecture}
The motivation for this conjecture is that  for $r=3$ we do not believe there to be counterexamples to Conjecture~\ref{Erdos} which are substantially different to those constructed in \cite{PokrovskiyCycles}.  

\subsubsection*{Partitioning high minimum degree graphs into monochromatic cycles}
The main application we gave of the theorems in this paper was to the area of partitioning coloured complete graphs into monochromatic cycles. For this application, our theorem was only one of the crucial ingredients. The other main ingredient we needed was Theorem~\ref{Letzter} about partitioning high minimum degree graphs into monochromatic subgraphs. It would interesting to better understand the number of monochromatic cycles to partition an $r$-coloured graph with minimum degree $\delta(G)$.

For $2$-coloured graphs, the minimum degree needed to get a partition by $2$ monochromatic cycles is well understood~\cite{Balogh,DeBiasioNelsen, LetzterMinDegree}. We make the following conjectures about how many cycles are needed in general for $2$ colours.
\begin{conjecture} \label{MinDegreeConjecture3}
The vertices of every sufficiently large $2$-edge coloured graph $G$ satisfying $\delta(G)> \frac{2}{3}|G|$ can be covered by $3$ disjoint monochromatic cycles.
\end{conjecture}

\begin{conjecture} \label{MinDegreeConjecture4}
The vertices of every sufficiently large $2$-edge coloured graph $G$ satisfying $\delta(G)> \frac{1}{2}|G|$ can be covered by $4$ disjoint monochromatic cycles.
\end{conjecture}
See Figure~\ref{figure4MinDegreeConjecture} for two graphs showing that the degree conditions in the above conjectures are essentially optimal. Notice that for $\epsilon>0$, the graph $K_{n,(1+\epsilon)n}$ has $\delta(G)\geq  (\frac{1}{2}-\epsilon)|G|$, and cannot be partitioned into less than $\epsilon |G|$ cycles (regardless whether the edges are coloured or not.) This shows that the degree threshold for partitioning graph into $k$ cycles has to be less than $\frac{1}{2}|G|-o(|G|)$ for any fixed $k$. Thus Conjecture~\ref{MinDegreeConjecture4} expresses the author's belief that the degree thresholds for partitioning $2$-coloured graphs into $k$ cycles should be roughly the same for $k\geq 4$.

There has already been some work on these conjectures.  Allen, B{\"o}ttcher,  Lang,  Skokan, and Stein~\cite{ABLSS} prove that for large $n$, any 2-edge-coloured graph $G$ on $n$ vertices and of
minimum degree $(2/3 +o(1))n$ can be partitioned into three
monochromatic cycles. It would also be interesting to understand how many monochromatic cycles are needed to partition $r$-edge-coloured graphs with high minimum degree. Here we don't even have a conjecture for what the best coloured graphs should be.

\begin{figure}
  \centering
    \includegraphics[width=0.6\textwidth]{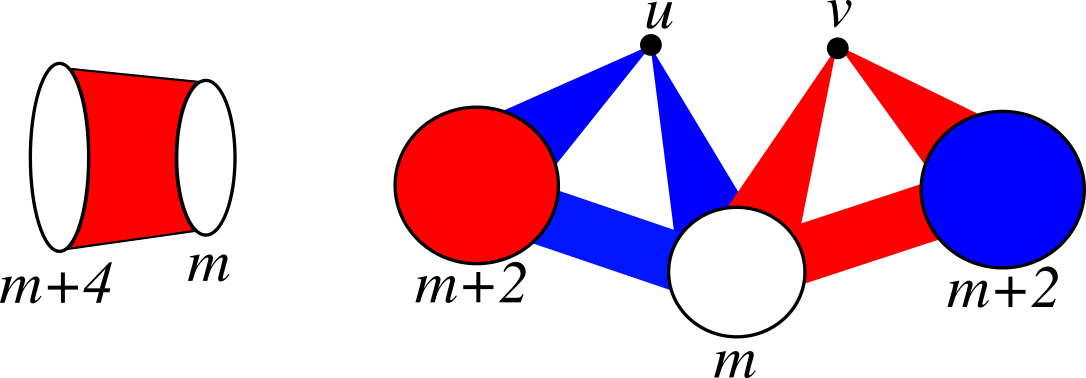}
  \caption{Coloured graphs showing that the bounds on the minimum degree in Conjectures~\ref{MinDegreeConjecture3} and~\ref{MinDegreeConjecture4} are essentially tight.} \label{figure4MinDegreeConjecture}
\end{figure}

\subsubsection*{Links to the Chv\'atal-Erd\H{o}s Theorem}
Let $\kappa(G)$ be the connectedness of a graph, and $\alpha(G)$ be the independence number of a graph.
Theorem~\ref{CycleMinDegreeConn}  shows that every graph has partition into a cycle $C$ and and induced subgraph $H$ of $G$ with $\Delta(H)\leq \frac{|H|}{\kappa(G)+1}+3$. By T\"uran's Theorem, graphs with small maximum degree cannot have very large independent sets. So if we applied Theorem~\ref{CycleMinDegreeConn} to a graph $G$ with small independence number, one would expect to obtain that the cycle $C$ is very long.

To make this precise, suppose that we have a graph $G$ with $(1+\epsilon)\alpha(G)\leq\kappa(G)+1$ for some $\epsilon>0$. Apply Theorem~\ref{CycleMinDegreeConn} to $G$  to get a partition into a cycle $C$ and and induced subgraph $H$ of $G$ with $\Delta(H)\leq \frac{|H|}{\kappa(G)+1}+3$. Rearranging  gives 
\begin{equation}\label{ChvatalErdosEq}
\kappa(G)+1\leq \frac{|H|+4(\kappa(G)+1)}{\Delta(H)+1}. 
\end{equation}
T\"uran's Theorem is equivalent to $\alpha(H)\geq \frac{|H|}{\partial(H)+1}$ where ``$\partial(H)$'' is the average degree of $H$. Using $\partial(H)\leq \Delta(H)$ and $\alpha(H)\leq \alpha(G)$ gives $\alpha(G)\geq \frac{|H|}{\Delta(H)+1}$. Combining this with (\ref{ChvatalErdosEq}) and $(1+\epsilon)\alpha(G)\leq\kappa(G)+1$ gives $|H|\leq 4(\kappa(G)+1)/\epsilon$. 

Thus we have shown that if $G$ is a graph with  $(1+\epsilon)\alpha(G)\leq\kappa(G)+1$, then $G$ has a cycle of length $\geq |G|-4(\kappa(G)+1)/\epsilon$. This isn't a very good result. By a theorem of Chvatal and Erd\H{o}s much more is true---such graphs are Hamiltonian.
\begin{theorem}[Chv\'atal and Erd\H{o}s, \cite{ChvatalErdos}]
Every graph $G$ with $\kappa(G)\geq \alpha(G)$ is Hamiltonian.
\end{theorem}
The above discussion shows that perhaps there is some relationship between Theorem~\ref{CycleMinDegreeConn} and the Chvatal-Erd\H{o}s Theorem. It would be interesting to investigate whether there is a common generalization of both theorems.

\subsubsection*{An alternative proof of the Bessy-Thomass\'e Theorem}
Theorem~\ref{CycleMinDegree} can be used to give an alternative proof of the Bessy~Thomass\'e Theorem. To do this we first need to 
classify graphs $H$ which satisfy $\Delta(H)\leq \frac12(|H|-1)$ but whose complements are not Hamiltonian. This is equivalent to classifying non-Hamiltonian graphs $G$ which have $\Delta(G)\geq \frac12(|G|-1)$. We can use the following theorem to do this.
\begin{theorem} [Nash-Williams, \cite{NashWilliams}]\label{NashWilliams}
Let $G$ be a $2$-connected graph satisfying 
$$\delta(G)\geq \max\left(\frac{1}{3}(|G|+2), \alpha(G)\right).$$ Then $G$ is Hamiltonian.
\end{theorem}

It is a simple exercise to use Theorem~\ref{NashWilliams} to classify non-Hamiltonian graphs $\Delta(G)\geq \frac12(|G|-1)$.
\begin{corollary}\label{stabledirac}
If $G$ is a graph satisfying $\delta(G)\geq \frac{1}{2}(|G|-1)$ then one of the following holds.
\begin{enumerate} [\normalfont(i)]
\item $V(G)$ can be partitioned into two sets $A$ and $B$ such that we have $|A| = |B|+1$, all the edges between $A$ and $B$ are present, and there are no edges within $A$.
\item There is a vertex $v\in G$ such that $V(G)-v$ can be partitioned into two sets $A$ and $B$ such that $|A|=|B|$, there are no edges between $A$ and $B$, and the subgraphs $G[A+v]$ and $G[B+v]$ are complete.
\item $G$ is Hamiltonian.
\end{enumerate}
\end{corollary}
Combining Theorem~\ref{CycleMinDegree} and Corollary~\ref{stabledirac} implies that the vertices of every graph $G$ have a partition into a cycle $C$ and a graph $H$ with $\overline H$ satisfying one of the structures (i) -- (iii) of Theorem~\ref{stabledirac}. When structures (i) or (ii) occur, we need quite a bit of extra work to obtain a partition into two cycles.  The full details can be found in~\cite{PokrovskiyAlternativeBessyThomasse}.

\subsubsection*{Graphs other than cycles}
A natural question is how the bounds in Theorems~\ref{CycleMinDegree} and~\ref{CycleMinDegreeConn} would change if we asked $C$ to be some graph other than a cycle.  
It is easy to see that the vertices of every graph can be partitioned into a matching and an independent set. This is the same as saying that every graph has a matching $M$ with $\Delta(G\setminus M)=0$.

If we replace $C$ by a path, then using Lemma~\ref{PathBipartite} it is easy to show that every graph has a path $P$ with $\Delta(G\setminus P)\leq |G|/2-1$.  To see that this is best possible notice that the graph formed from two disjoint copies of  $K_m$ has $\Delta(G\setminus P)\geq |G|/2-1$ for every path $P$. It would be interesting to find more analogues of Theorems~\ref{CycleMinDegree} and~\ref{CycleMinDegreeConn} when the cycle is replaced by other kinds of graphs.

\subsubsection*{Other degree conditions}
In this paper we investigated partitions of graphs into a cycle $C$ and an induced subgraph $H$ with small maxmimum degree. 
One of the motivations for researching this was to find a strengthening of the Bessy-Thomass\'e Theorem of the form ``the vertices of every graph can be partitioned into a cycle and an induced subgraph $H$ with $\overline{H}$ Hamiltonian for some \emph{natural} reason.''
The ``natural reason'' which we tried to achieve in this paper was the condition in Dirac's  Theorem. We failed to achieve this since the graph in Figure~\ref{FigureGeneralGraphExtremal} doesn't have a partition into a cycle and a complement of a Dirac graph. However there are many other natural conditions for Hamiltonicity which we could try and get $H$ to have. In particular, over the last decades there have been a huge number of generalizations of Dirac's Theorem proved. For each of these one could ask whether the appropriate strengthening of the Bessy-Thomass\'e Theorem is true.

One of the earliest and most famous generalizations of Dirac's Theorem is the following theorem of Ore.
\begin{theorem}[Ore~\cite{Ore}]\label{OreTheorem}
Let $G$ be a graph in which $d(u)+d(v)\geq n$ holds for any pair of nonadjacent vertices $u$ and $v$. Then $G$ is Hamiltonian.
\end{theorem}

Notice that the graphs in Figure~\ref{FigureGeneralGraphExtremal} have a partition into a cycle (the edge $ux$), and an induced graph $H$ whose complement satisfies the assumptions of Theorem~\ref{OreTheorem}. We conjecture that such a partition exists for all $G$.
\begin{conjecture}\label{ConjectureOre}
Every graph has a cycle $C$ such that for any adjacent vertices $u,v\not\in C$  we have $d_{G\setminus C}(u)+d_{G\setminus C}(v)\leq |G\setminus C|-2$.
\end{conjecture}

As a more open ended question, we believe that graphs like those in Figure~\ref{FigureGeneralGraphExtremal} should be quite limited.
\begin{problem}
Classify all graphs $G$ which don't have a  cycle $C$ satisfying $\Delta(G\setminus C)\leq \frac12|G\setminus C|-1.$
\end{problem}
It is quite possible that an approach like the one we used for Theorem~\ref{CycleMinDegree} can be used to solve the above problem. However, given how much case analysis was needed for Theorem~\ref{CycleMinDegree}, extending it to classify all extremal cases could be quite tricky.

\bigskip\noindent
\textbf{Acknowledgment}

\smallskip\noindent
The author would like to thank his supervisors Jan van den Heuvel and Jozef Skokan for their advice and discussions. He would like to thank Pedro Vieira for suggesting that Conjecture~\ref{ConjectureOre} might be true.
This research was partly supported by the LSE postgraduate research studentship scheme and the Methods for Discrete Structures, Berlin graduate school (GRK 1408), and  SNSF grant 200021-149111.

\bibliography{pathpartition}
\bibliographystyle{abbrv}

\end{document}